\documentclass[pdflatex,sn-mathphys]{sn-jnl}

\jyear{2021}%

\usepackage{mathtools}
\theoremstyle{thmstyleone}%
\newtheorem{theorem}{Theorem}[section]
\newtheorem{proposition}[theorem]{Proposition}%
\newtheorem{lemma}[theorem]{Lemma}

\newtheorem{corollary}[theorem]{Corollary}

\theoremstyle{thmstyletwo}%
\newtheorem{example}{Example}%
\newtheorem{remark}{Remark}%

\theoremstyle{thmstylethree}%
\newtheorem{definition}{Definition}[section]%

\begin{document}

\title[Scalarization in Set-Valued Optimization Problems]{Scalarization and Robustness in Uncertain Set-valued Optimization }


\author*[1]{\fnm{Abhik} \sur{Digar}}\email{abhikdigar@gmail.com}


\author[2]{\fnm{Kuntal} \sur{Som}}\email{kuntal24.som@gmail.com}\equalcont{These authors contributed equally to this work.}

\affil*[1]{\orgdiv{Department of Mathematics \& Statistics}, \orgname{Indian Institute of Technology Kanpur}, \orgaddress{\street{Kalyanpur}, \postcode{208016}, \state{Uttar Pradesh}, \country{India}}}

\affil[2]{\orgdiv{Department of Mathematics}, \orgname{Indian Institute of Technology Jodhpur}, \orgaddress{\city{Jodhpur}, \postcode{342030}, \state{Rajasthan}, \country{India}}}


\abstract{In this paper, we remark on the published paper ``Treatment of Set-Valued Robustness via Separation and Scalarization'' \cite{Das2024}, which deals with the robust solution to an uncertain constrained set-valued optimization problem via scalarization methods. We show many inconsistencies in the results of the above-mentioned paper. We improve most of these results. In the process, we introduce some new concepts of robust solutions for uncertain set-valued optimization problems. We also improve some results on scalarization methods applicable to set-valued optimization.}





\maketitle

\section{Introduction}

Set-valued optimization is a rapidly growing research area with lots of applications in diverse fields \cite{borwein1977multivalued,corley1988optimality,DK4,Akhtar,Gofert}. In the early days of development, a `vector approach' of solution to a set-valued optimization problem was prominent \cite{borwein1977multivalued,corley1987existence,corley1988optimality,Luc89}. However, some shortcomings of this approach were pointed out in the late 1990s by Kuroiwa et al. \cite{DK}, which led to the development of the `set-relation approach' of solutions to a set-valued optimization problem \cite{KurTanTru97,DK4}. Since then, research in set-valued optimization in the set-relation approach has been steadily rising along multiple avenues, such as study of existence results for solutions of set-valued optimization problems and optimality conditions \cite{DK4,alonso2005set,hernandez2007existence,kuroiwa2003weightedcriteriaforexistence,vetrivelkuntalbookchater}, the study of scalarization and Ekeland's variational principle \cite{minimalelementHamel2006,hernandez2007nonconvex,guttierez2015scalarizationsurvey,Pradeeprecentsurvey,gerth1990nonconvex,Kobis2016}, the study of well-posedness properties of set-valued optimization problems \cite{Zhang2009,Gutierrez2012,Han2017,Pap2,Pap3} etc. to name a few. Also, set-valued optimization has been used as modeling tool in risk theory \cite{HamHey10,HM1}, in behavioral sciences \cite{Boris2010}, in games with incomplete information or multi-dimensional pay-offs \cite{H}, in bilevel programming problems \cite{Zemkohoillposed,pileckathesis,Pap6} and in robust vector and set optimization problems \cite{IDE,CreKurRoc17,Pap1}. 

Recently, robustness for an uncertain set-valued optimization problem with set-valued constraints has been studied in \cite{Das2024}, where various scalarization methods have been used to characterize different robust solutions. The paper uses scalarization methods proposed in \cite{Kobis2016,hernandez2007nonconvex} to derive the main results. However, on close inspection, we found inconsistencies in multiple results in the above mentioned paper. We point out those inconsistencies and propose modifications that seem necessary. In the process, we introduce some new concepts of robust solutions for uncertain set-valued optimization problems that arise naturally. We also improve upon a result on the scalarization functions from \cite{Kobis2016}. This is the aim of this paper. 

The structure of the paper is as follows. In Section \ref{prelim}, we introduce a set-valued optimization problem and solution notions in vector and set-relation approaches. We recall basic notations and other preliminaries from the set-valued optimization literature that are necessary for the rest of the paper. We also introduce some new set order relations that apply to the union of sets. In Section \ref{scalarization in set valued optimization}, we recall various scalarization functions and their application in classifying set order relations and their implication in set-valued optimization literature. We rectify some results in \cite{Das2024} and improve one result from \cite{Kobis2016}.  In Section \ref{uncertain set valued problem}, we introduce an uncertain set-valued optimization problem with set-valued constraints that is taken from \cite{Das2024}. We recall various notions of robust solutions and introduce some new robust solution concepts. We then study various necessary and sufficient criteria for characterizing these robust solutions via scalarization functions.

\section{Preliminaries}\label{prelim}

A set-valued optimization problem in the most general form looks like: 
\begin{gather} \label{setvaluedoptimization}
    \text{minimize } G(x) \\ \text{subject to } x \in S, \nonumber
\end{gather} 
where $G:X \rightarrow \mathcal{P}(Y)$ is a map, $X$ and $Y$ are topological vector spaces, and $S\subseteq X$ a nonempty constraint set. The nonempty subsets of $Y$ is denoted by $\mathcal{P}(Y)$. 
Since optimization needs order structure, it is assumed that 
$Y$ is partially ordered by a nonempty closed convex pointed cone $K$, where the induced order $\leq_K$ is defined as: for $z_1,z_2\in Z$, $z_1\leq_K z_2$ if and only if $z_2-z_1\in K$. For a set $A\subseteq Y$, an element $\hat{a}\in A$ is called minimal with respect to the order relation $\leq_K$, if for any $a\in A$ it holds that $a\leq_K \hat{a}$, then $\hat{a}=a$.

The foundation of set-valued optimization can be attributed to the works of Borwein \cite{borwein1977multivalued}, Postolic\u{a} \cite{postolm1986vectorial}, and Corley \cite{corley1987existence,corley1988optimality}. A thorough exploration of this topic is available in the book \cite{Luc89}. The solution concept considered in these studies was later referred to as the `vector approach,' as it generalizes the notion of solutions available in vector optimization problems \cite{Luc89}.

 \begin{definition}[\cite{Luc89}]
 Consider the problem {\normalfont(\ref{setvaluedoptimization})}. The pair $(x^0,y^0) \in S\times Y$ is called a \textit{vector solution} to {\normalfont(\ref{setvaluedoptimization})} if $y^0 \in G(x^0)$ and $y^0$ is a minimal point of $G(S)=\underset{x\in S} {\bigcup} G(x)$, where the minimality is with respect to the order relation $\leq_K$.
 \end{definition}

The vector approach determines the minimal elements of the total image set $G(S)$ with respect to the underlying vector order $\leq_K.$ While this is a widely used solution concept, its drawback is that it considers only a single optimal point within the image set of a solution. In practical scenarios, this perspective may not always provide an accurate representation. For instance, in a soccer league, a team with one exceptional player but otherwise below-average teammates may not be considered truly strong. Recognizing the need for a more comprehensive method of comparing sets, the `set-relation' approach was introduced and later popularized by Kuroiwa et al. (see \cite{KurTanTru97,DK4}). This approach revolutionized set-valued optimization, paving the way for an entirely new research direction.
 
In the set-relation approach, sets are compared using set-order relations. For two nonempty subsets $A,B \subseteq Y$, consider the following set order relations:
\begin{itemize} 
	\item $A \leq ^l _K B $ if and only if $B \subseteq A+K$.
	\item $A \leq ^u _K B $ if and only if $A \subseteq B-K$.
        \item $A \leq ^s _K B $ if and only if $A \subseteq B-K$ and $B \subseteq A+K$.
\end{itemize} 
These set order relations are preorders (that is, reflexive and transitive), and with respect to each such order, the set $\mathcal{P}(Y)$ is a preordered space. Based on the set order relations introduced above, the following are a few notions of the `set-relation approach' of solutions for (\ref{setvaluedoptimization}) as has been introduced in \cite{DK4}.
\begin{definition}
Consider the problem {\normalfont(\ref{setvaluedoptimization})}. A point $x^0 \in S$ is called 
\begin{itemize}
\item[(i)] an \textit{$l$-minimal} (also called \textit{$l$-type}) solution to {\normalfont(\ref{setvaluedoptimization})} if for any $x \in S$ such that $G(x) \leq^l_K$ $G(x^0)$ we have $G(x^0) \leq^l_K$ $G(x)$.

\item[(ii)] a \textit{$u$-minimal} (also called \textit{$u$-type}) solution to {\normalfont(\ref{setvaluedoptimization})} if for any $x \in S$ such that $G(x) \leq^u_K$ $G(x^0)$ we have $G(x^0) \leq^u_K$ $G(x)$.
\end{itemize}
\end{definition}

In the paper \cite{Das2024}, if we carefully see, though $l$-type and $u$-type relations have been used to define robust solutions for uncertain set-valued optimization problem, actually the results are derived for stronger order relations. This motivates us to define the following new set-order relation for sets that are unions of a collection of sets.

\begin{definition}\label{Def:Lower} Let $Y$ be a topological vectore space ordered by a nonempty closed convex pointed cone $K$. Let $\{P_\gamma: \gamma \in \Gamma\}$ and $\{Q_\lambda: \lambda\in \Lambda\}$ be two collections of non-empty subsets of $Y.$ Set $P=\underset{\gamma\in \Gamma}{\bigcup} P_{\gamma}$ and $Q=\underset{\lambda\in \Lambda}{\bigcup} 
 Q_\lambda.$ We define
 \begin{itemize}
     \item $P\leq_{K}^{L} Q$ if for every $\lambda \in \Lambda,$ there exists $\gamma\in \Gamma$ such that $P_\gamma\leq_{K}^{l} Q_\lambda.$
     \item $P\leq_{K}^{U} Q$ if for every $\gamma \in \Gamma,$ there exists $\lambda\in \Lambda$ such that $P_\gamma \leq_{K}^{u}Q_\lambda.$
     \item $P\leq_{K}^{S} Q$ if for every $\lambda \in \Lambda,$ there exists $\gamma\in \Gamma$ such that $P_\gamma\leq_{K}^{l} Q_\lambda$ and  for every $\gamma' \in \Gamma,$ there exists $\lambda'\in \Lambda$ such that $P_{\gamma'} \leq_{K}^{u} Q_{\lambda'}.$
 \end{itemize} 
\end{definition}

Notice that if $\Gamma$ and $\Lambda$ are singleton index sets, say $\Gamma =\{\gamma\}$ and $\Lambda =\{\lambda\}$,  then trivially $P\leq_{K}^{\mu} Q$ if and only if $P_\gamma\leq_{K}^{\nu} Q_\lambda$ holds true, where $(\mu, \nu)\in \{(L, l), (U, u), (S, s)\}.$ Henceforth we will assume that $Y$ is a topological vector space ordered by a nonempty closed convex pointed cone $K\subseteq Y$.

\begin{proposition}\label{Ll relations}
  Let $\{P_\gamma: \gamma \in \Gamma\}$ and $\{Q_\lambda: \lambda\in \Lambda\}$ be two collections of non-empty subsets of $Y.$ Set $P=\underset{\gamma\in \Gamma}{\bigcup} P_{\gamma}$ and $Q=\underset{\lambda\in \Lambda}{\bigcup} 
 Q_\lambda.$ If $P\leq_{K}^{L} Q,$ then $P\leq_{K}^{l} Q.$  
\end{proposition}
\begin{proof}
    Assume that $P\leq_{K}^{L} Q.$ Let $q\in Q.$ Then there exists $\lambda\in \Lambda$ such that $q\in Q_\lambda.$ By the hypothesis, there exists $\gamma\in \Gamma$ for which $Q_\lambda\subseteq P_\gamma+K\subseteq P+K.$ Therefore, $q\in P+K.$ Consequently, $Q\subseteq P+K.$ This completes the proof.
\end{proof}
Thus $\leq_{K}^{L}$ is a set order relation stronger than $\leq_{K}^{l}$ when the union of sets is considered. The ordering or parametrization is important while defining the $L$-type order relation. We illustrate this via an example.

 \begin{example}\label{Example:l but not L}
Let $Y=\mathbb{R}^2$ and $K=\mathbb{R}^2_{+}.$
    Consider the triangles $A_1, A_2, B_1$ and the trapezium $B_2$ as follows.
    \begin{eqnarray*}
        A_1 &=& \triangle \left( (0,0), \, (-1,0), \, \left(-\frac{1}{2}, -\frac{1}{2}\right) \right),\\
        A_2 &=& \triangle \left( (0,0), \, \left(-\frac{1}{2}, -\frac{1}{2}\right), \, (0,-1) \right),\\
        B_1 &=& \triangle \left( (0,0), \, \left(-\frac{3}{4},0\right), \, \left(0,-\frac{3}{4}\right) \right),\\
        B_2 &=& \text{trapezium} \left( (-1,0), \, \left(-\frac{3}{4},0\right), \, \left(0,-\frac{3}{4}\right), \, (0,-1) \right).
    \end{eqnarray*}

Here $\triangle (x_1, x_2, x_3)$ denotes the triangle with the vertices $x_1, x_2, x_3$ and trapezium $(y_1, y_2, y_3, y_4)$ denotes the trapezium with the vertices $y_1, y_2, y_3, y_4.$ Clearly, $A_1\cup A_2=B_1\cup B_2$ and hence $(B_1\cup B_2)\subseteq (A_1\cup A_2)+K.$
\begin{figure}[H]
    \centering
    \includegraphics[width=0.45\linewidth]{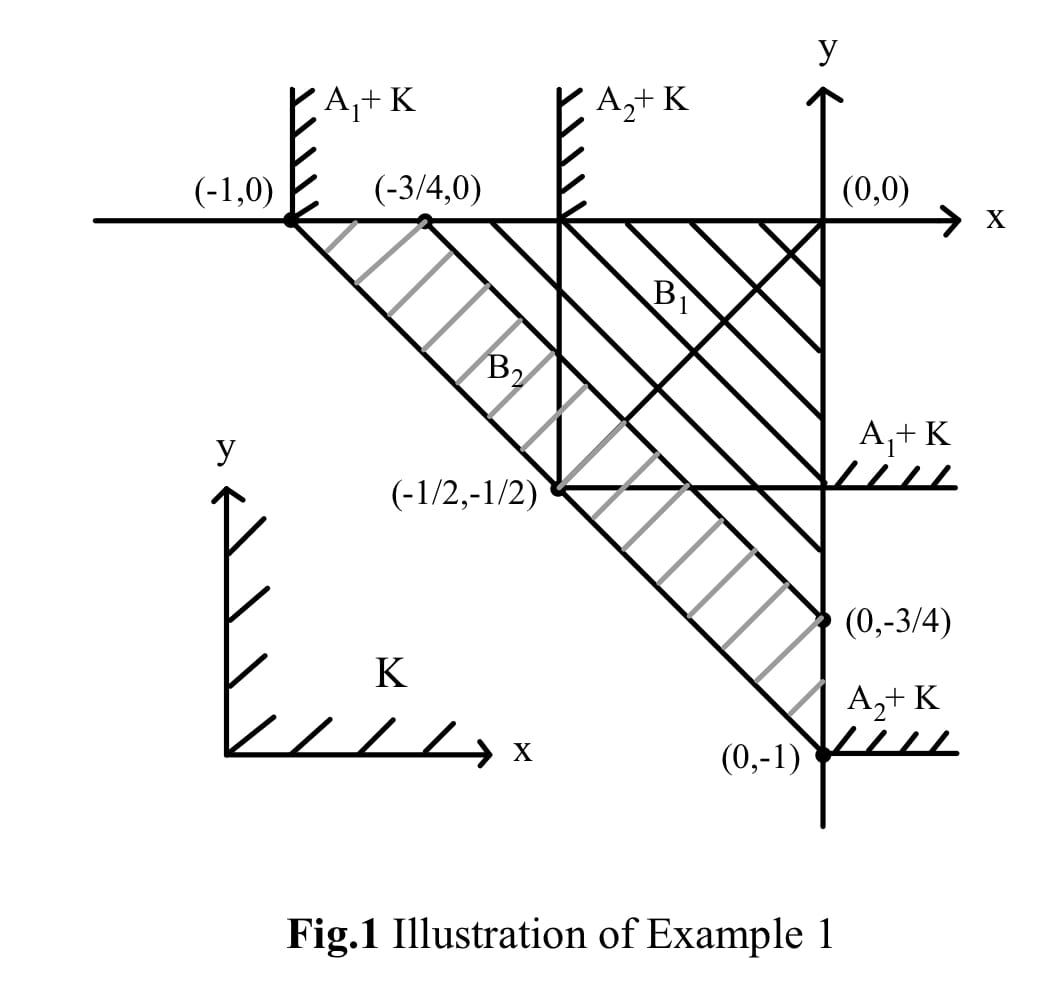}
    \label{fig:enter-label}
\end{figure}
From the figure, it can be seen that $B_i\nsubseteq A_j+K$ for any $1\leq i,j\leq 2.$ \qed
\end{example}

  The above example illustrates that, though $A=A_1\cup A_2 = B_1\cup B_2=B $, $A$ and $B$ may not be linked via an $L$-type set order relation. So, it is not just the union, rather how one takes the union is important. The following result is easy to see for $U$-type set order relation.
  \begin{proposition}\label{Uu relation}
  Let $\{P_\gamma: \gamma \in \Gamma\}$ and $\{Q_\lambda: \lambda\in \Lambda\}$ be two collections of non-empty subsets of $Y.$ Set $P=\underset{\gamma\in \Gamma}{\bigcup} P_{\gamma}$ and $Q=\underset{\lambda\in \Lambda} {\bigcup}
 Q_\lambda.$ If $P\leq_{K}^{U} Q,$ then $P\leq_{K}^{u} Q.$  
\end{proposition}
We can again show that $\leq_{K}^{U}$ is a stronger set order relation than $\leq_{K}^{u}$ when the union of sets is considered. The following example illustrates that this distinction is strict.
\begin{example}\label{Example:u but not U}
    Let $Y=\mathbb{R}^2$ and $K=\mathbb{R}^2_{+}.$
    Consider the triangles $A_1, A_2, B_1$ and the trapezium $B_2$ as follows.
    \begin{eqnarray*}
        A_1 &=& \triangle \left( (0,0), \, \left(\frac{3}{4},0\right), \, \left(0,\frac{3}{4}\right) \right),\\
        A_2 &=& \text{trapezium} \left( (1,0), \, \left(\frac{3}{4},0\right), \, \left(0,\frac{3}{4}\right), \, (0,1) \right),\\
        B_1 &=& \triangle \left( (0,0), \, (1,0), \, \left(\frac{1}{2}, \frac{1}{2}\right) \right),\\
        B_2 &=& \triangle \left( (0,0), \, \left(\frac{1}{2}, \frac{1}{2}\right), \, (0,1) \right).
    \end{eqnarray*}

The notations used here are in accordance with those defined in Example \ref{Example:l but not L}.

Since $A_1\cup A_2=B_1\cup B_2$, it is true that $(A_1\cup A_2)\subseteq (B_1\cup B_2)-K.$

\begin{figure}[H]
    \centering
    \includegraphics[width=0.45\linewidth]{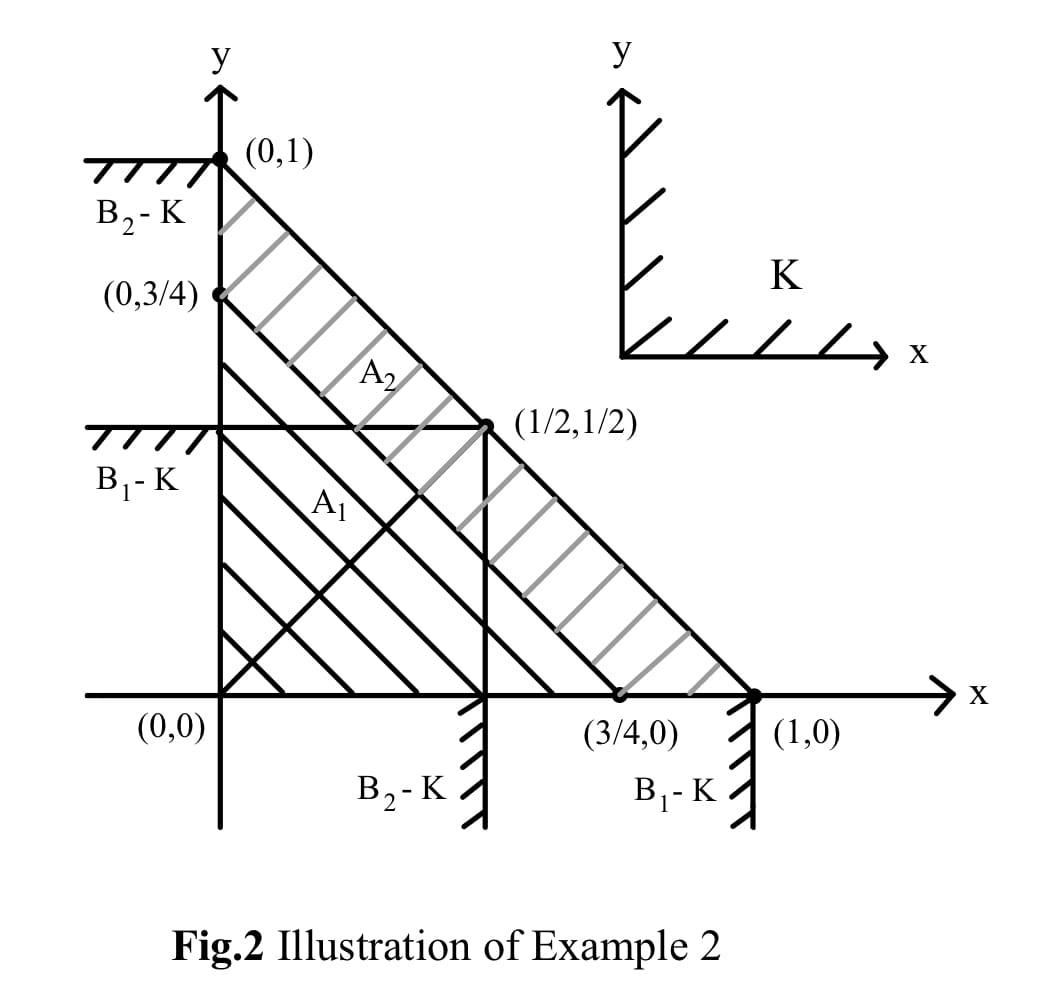}
    \label{fig:enter-label}
\end{figure}
From the figure, it is evident that $A_i\nsubseteq B_j-K$ for any $1\leq i, j\leq 2.$
\end{example}

Using Propositions \ref{Ll relations}, \ref{Uu relation} we can derive a relationship between $\leq_K^S$ and $\leq_K^s$ as we mention below, without proof. 
\begin{proposition}
  Let $\{P_\gamma: \gamma \in \Gamma\}$ and $\{Q_\lambda: \lambda\in \Lambda\}$ be two collections of non-empty subsets of $Y.$ Set $P=\underset{\gamma\in \Gamma}{\bigcup} P_{\gamma}$ and $Q=\underset{\lambda\in \Lambda} {\bigcup}
 Q_\lambda.$ If $P\leq_{K}^{S} Q,$ then $P\leq_{K}^{s} Q.$
 
\end{proposition}

As one can expect, $\leq_{K}^{S}$ is stronger than $\leq_{K}^{s}$ when the union of sets are considered. The following is one such illustrious example.
\begin{example}\label{Example:u but not U}
Consider the sets $A_1, A_2, B_1$ and $B_2$ as in Example \ref{Example:l but not L}. Set $A=A_1\cup A_2$ and $B=B_1\cup B_2.$ Since $A=B$, it follows that $A\leq_{K}^{s} B.$ However, Example \ref{Example:l but not L} suggests that $A\nleq_{K}^{L} B.$ Hence $A\nleq_{K}^{S} B.$
   
\end{example}
In the paper \cite{Das2024}, the authors have established set order relations between two unions of sets (see Propositions 3.1, 3.2 and 3.3 in \cite{Das2024}). However, there are ambiguities, as we point out below. First, let us recall Proposition 3.1 from \cite{Das2024}.
\begin{proposition}
     Let $\{P_\gamma: \gamma \in \Gamma\}$ and $\{Q_\lambda: \lambda\in \Lambda\}$ be two collections of non-empty subsets of $Y.$ Set $P=\underset{\gamma\in \Gamma}{\bigcup} P_{\gamma}$ and $Q=\underset{\lambda\in \Lambda} {\bigcup} Q_\lambda.$  Suppose for each $\lambda \in \Lambda,$ there exists $p_\lambda\in P$ such that $Q_\lambda\subseteq p_\lambda+K$. Then $P\leq_K^l Q$ if and only if for each $\lambda\in \Lambda,$ there exists $\bar{\gamma}\in \Gamma$ such that $P_{\bar{\gamma}}\leq_K^l Q_\lambda$.
\end{proposition}
\begin{remark}
    It is trivial to see that once it is assumed that for each $\lambda \in \Lambda,$ there exists $p_\lambda \in P$ such that $Q_\lambda \subseteq p_\lambda + K$, then $Q=\underset{\lambda\in \Lambda}{\bigcup} Q_\lambda \subseteq \underset{\lambda \in \Lambda}{\bigcup} (p_\lambda + K) \subseteq P+K$. Also under the same assumption, for each $\lambda \in \Lambda,$ one can take any $P_{\bar{\gamma}}$ to which  $p_\lambda$ belongs to for concluding that $P_{\bar{\gamma}}\leq_K^l Q_\lambda$. Thus, the `if and only if' in the conclusion does not make any sense. Similar things hold for Propositions 3.2 and 3.3 in \cite{Das2024} as well. 
\end{remark}

We now move to study scalarization in set-valued optimization.

\section{Scalarization in set-valued optimization}\label{scalarization in set valued optimization}
Scalarization is an important tool to study set-valued optimization problems. We recall some standard scalarization functions from the literature \cite{Kobis2016,hernandez2007nonconvex,guttierez2015scalarizationsurvey}. The following function is one of the important nonlinear scalarization functions used for multi-objective optimization as has been studied in \cite{gerth1990nonconvex} (also see \cite{Das2024,hernandez2007nonconvex,Kobis2016}).

\begin{definition} 
Let $Y$ be a topological vector space partially ordered by a nonempty proper closed convex pointed cone $K\subseteq Y$ and let $e\in K\setminus \{0\}.$
   Define the function $z^{e, K}: Y\to [-\infty, \infty]$ by $$z^{e, K}(y)=\inf\{t\in \mathbb{R}: y\in te-K\}.$$
\end{definition}
\begin{remark}
    Note that the cone in the above definition is not assumed to be solid, that is, it may have an empty interior. However, at places, we will assume that the interior is nonempty and will explicitly mention as applicable. Also, we have assumed the cone to be pointed. This pointed assumption is in line with \cite{Das2024,hernandez2007nonconvex}. The paper \cite{Kobis2016} used the above scalarization function where the cone has not been assumed to be pointed. However, without the cone being pointed, the function $z^{e,K}$ can be improper. For example, take the cone $K=\{(x,0)\in \mathbb{R}^2\;: x\in \mathbb{R}\} $, that is $K$ is the entire $x$-axis. Then for any $e\in K\setminus \{0\}$, $$z^{e,k}(y)=\begin{cases}
        +\infty & \text{if } y\notin K;\\
        -\infty & \text{if } y\in K.
    \end{cases} $$
In fact one can easily show that when $K$ is not pointed, for every $e\in K\cap (-K), \;z^{e,k}(y)\in \{-\infty,\infty\} \text{ for all }y\in Y$. So the pointedness is important.
\end{remark}

Generalizations of $z^{e, K}$ have been proposed in \cite{Kobis2016} to study set-valued optimization problems. By $\mathcal{P}_{K}(Y)$ we denote the collection of all nonempty $K$-proper subsets of $Y$, that is, $\mathcal{P}_{K}(Y)=\{A\subseteq Y\;:\; A\neq \emptyset \text{ and } A+K\neq Y\}$.
\begin{definition}
Let $Y$ be a topological vector space and $K\subseteq Y$ be a nonempty proper closed convex pointed cone and let $e\in K\setminus \{0\}$.
\begin{itemize}
    \item[(i)] The function $\mathbb{Z}_1^{e, K}: \mathcal{P}_{K}(Y) \times \mathcal{P}_{K}(Y)\to [-\infty, \infty]$ is defined by $$\mathbb{Z}_1^{e,K}(P,Q)=\adjustlimits \sup_{y\in Q} \inf_{x\in P} z^{e, K}(x-y).$$
     \item[(ii)] The function $\mathbb{Z}_2^{e, K}: \mathcal{P}_{K}(Y) \times \mathcal{P}_{K}(Y)\to [-\infty, \infty]$ is defined by $$\mathbb{Z}_2^{e,K}(P,Q)=\adjustlimits \sup_{x\in P} \inf_{y\in Q} z^{e, K}(x-y).$$
    
\end{itemize}  
\end{definition}

These functions appeared in \cite{Kobis2016}, but not in the definition form. When $\operatorname{int } (K)$ is nonempty, the following scalarization functions have been defined in \cite{hernandez2007nonconvex}. 
\begin{definition}
Let $Y$ be a topological vector space and $K\subseteq Y$ be a nonempty proper closed convex pointed solid cone. Let $e\in -\operatorname{int } (K).$
\begin{itemize}
    \item For $a,b\in Y$, define $\phi_{e,a}(b)=\inf\{t\in \mathbb{R}: b\in te+a+K\}$.
    \item For $b\in Y$, and $A\subseteq Y$ define $\phi_{e,A}(b)=\inf\{t\in \mathbb{R}: b\in te+A+K\}$.
    \item For $A,B\subseteq Y$, define $G_e(A,B)=\underset{b\in B}{\sup}~ \phi_{e, A} (b)$.
\end{itemize}

\end{definition}

    It is easy to see that for any $e\in \operatorname{int } (K)=-\operatorname{int } (-K),$
            \begin{eqnarray*}
                \phi_{e,a}(b)&=&\inf\{t\in \mathbb{R}: b\in te+a-K\}\\
                &=& \inf\{t\in \mathbb{R}: b-a\in te-K\}=z^{e,K} (b-a)\\
                &=& \phi_{e,0}(b-a)\\
               &=&  \inf\{t\in \mathbb{R}: b-a\in te-K-K\}~~({{\text{\small{since K is a convex cone,}}}~K+K=K})\\
                &=& \phi_{e,-K}(b-a)\\
                &=& \inf\{t\in \mathbb{R}: a-b\in t(-e)+K+K\}~~=z^{-e, -K} (a-b).
            \end{eqnarray*}

Similarly, for any $e\in - \operatorname{int } (K),$ $\phi_{e,a}(b)=z^{e,-K} (b-a)$ $=z^{-e, K} (a-b).$
           
\noindent Thus $G_e(A,B)=\underset{b\in B}{\sup}~ \phi_{e, A} (b)=\underset{b\in B}{\sup}~ \underset{a\in A}{\inf} \phi_{e, K} (b-a)=\underset{b\in B}{\sup}~ \underset{a\in A}{\inf} z^{e,-K} (b-a)$ $=\underset{b\in B}{\sup}~ \underset{a\in A}{\inf} z^{-e,K} (a-b)=\mathbb{Z}_1^{-e,K}(A, B)$.

We now collect some results from \cite{hernandez2007nonconvex} that connect these scalarization functions with the set order relations.
\begin{lemma}\label{Lem1} Let $Y$ be a topological vector space, $K\subseteq Y$ be a nonempty proper pointed closed convex solid cone. Let $P, Q\in \mathcal{P}_{K}(Y).$ Then
    \begin{itemize}
        \item[\normalfont{(i)}]  If $P+K$ is closed then ${G}_{e}(P,P)=0$ $\text{for all }~e\in -\operatorname{int } (K).$
       
        \item[\normalfont{(ii)}] Assume that $P+K$ is closed in $Y$. Then $P\leq_{K}^{l} Q$ if and only if ${G}_{e}(P, Q)\leq 0~\text{for all }~e\in -\operatorname{int } (K).$

         \item[\normalfont{(iii)}] Let $P+K$ is closed, $r\in \mathbb{R}$ and $e\in -\operatorname{int } (K).$ Then ${G}_{e}(P,Q)\leq r$ if and only if $Q\subseteq re+ P+K.$
    \end{itemize}
       
\end{lemma}        
When the cone is not assumed to be solid or pointed, the following characterizations of different set order relations in terms of scalarization functions are taken from \cite{Kobis2016}.

\begin{theorem}\label{Thm1:Kobis2016}
    Let $P, Q\subseteq Y$  be two nonempty sets and let $K\subseteq Y$ be a nonempty proper closed convex cone. Then
    $$P\leq_{K}^{l} Q \Longrightarrow \mathbb{Z}_1^{e, K} (P,Q)\leq 0 ~\text{for all } e\in K\setminus \{0\}.$$ Further assume that there exists $e_0\in K\setminus \{0\}$ such that $\underset{p\in P}{\inf}~ z^{e_0, K} (p-q)$ is attained for all $q\in Q.$ Then
    $$P\leq_{K}^{l} Q \Longleftrightarrow \mathbb{Z}_1^{e_0, K} (P,Q)\leq 0\Longleftrightarrow \sup_{e\in K\setminus \{0\}} \mathbb{Z}_1^{e, K} (P,Q)\leq 0.$$
\end{theorem}
\begin{theorem}\label{Thm2:Kobis2016}
    Let $P, Q\subseteq Y$  be two nonempty sets and let $K\subseteq Y$ be a nonempty proper closed convex cone. Then
     $$P\leq_{K}^{u} Q \Longrightarrow \mathbb{Z}_2^{e, K} (P,Q)\leq 0~\text{for all } e\in K\setminus \{0\}.$$ Further assume that there exists $e_0\in K\setminus \{0\}$ such that $\underset{q\in Q}{\inf}~ z^{e_0, K} (p-q)$ is attained for all $p\in P.$ Then
    $$P\leq_{K}^{u} Q \Longleftrightarrow \mathbb{Z}_2^{e_0, K} (P,Q)\leq 0\Longleftrightarrow \underset{e\in K\setminus \{0\}}{\sup} \mathbb{Z}_2^{e, K} (P,Q)\leq 0.$$
\end{theorem}

We observe that the attainment property in Theorems \ref{Thm1:Kobis2016} and \ref{Thm2:Kobis2016} is not necessary. We illustrate it with an example.

    \begin{example}\label{Counterexample to Kobis}
        Let $Y=\mathbb{R}^2,$ $K=\{(x,0): x\geq 0\},$ $A=\{(x,0): x\leq 0\}$ and $B=K.$ Then $B-K=\{(x,0): x\in \mathbb{R}\}$ and hence $A\subset B-K.$ Let $a=(a_1, 0)\in A$ and $b=(b_1, 0)\in B$ be any two elements. Let $e\in K\setminus \{0\}$. Then $e=(c, 0)$ for some $c>0.$ We get
        \begin{eqnarray*}
          a-b\in te-K &\Longleftrightarrow & a_1-b_1=tc-k~\text{for some } k\geq 0\\
          &\Longleftrightarrow & t\geq \frac{a_1-b_1}{c}.
        \end{eqnarray*}
This shows that $z^{e,K}(a-b)=\frac{a_1-b_1}{c}$ and $\underset{b\in B}{\inf}~z^{e,K}(a-b)=-\infty.$ It follows that $\underset{a\in A}{\sup}~\underset{b\in B}{\inf}~z^{e,K}(a-b)=-\infty.$ In fact, $\underset{e'\in K\setminus \{0\}}{\sup}~\underset{a\in A}{\sup}~\underset{b\in B}{\inf}~z^{e',K}(a-b)=-\infty.$ Fix $a'\in A.$ Then we can't find any $b'\in B$ such that $z^{e', K}(a'-b')=-\infty$ for any $e'\in K\setminus \{0\}.$
        \end{example}
Thus we propose improvement in Theorem \ref{Thm1:Kobis2016} and \ref{Thm2:Kobis2016} (see Theorems 3.3 and 3.8 in \cite{Kobis2016}). The proof of the improved results will require the following lemma.
\begin{lemma}
    Let $K\subseteq Y$ be a nonempty proper closed convex cone and $P\in \mathcal{P}_K(Y)$ such that $P+K$ is closed. Let $e\in K\setminus \{0\}.$ Then, $$P+K=\underset{\alpha>0}{\bigcap}\left(P+K-\alpha e\right).$$
\end{lemma}
\begin{proof}
Let $y\in P+K.$ Then there exists $p\in P$ such that $y-p\in K.$ Since $K$ is a convex cone, we get $\alpha e\in K$ for all $\alpha>0.$ That is, $y-p+\alpha e\in \alpha e+K \subseteq K$ for all $\alpha>0.$ Consequently, $y\in p+K-\alpha e$ $\subseteq P+K-\alpha e$ for all $\alpha>0.$

For the converse part, assume on the contrary that $\underset{\alpha>0}{\bigcap}\left(P+K-\alpha e\right) \nsubseteq P+K.$ Then there exists $x\in \underset{\alpha>0}{\bigcap}\left(P+K-\alpha e\right)$ but $x\notin P+K.$ Since $P+K$ is closed, there exists an open set $U$ in $Y$ such that $x\in U$ and $U\cap (P+K)=\emptyset$. Since $Y$ is a topological vector space, addition is continuous and hence we can find a neighbourhood $\mathcal{O}$ of the origin of $Y$ such that $x+\mathcal{O}\subseteq U.$ Moreover, there exists a neighbourhood $V$ of the origin such that $V$ is balanced and $V\subset \mathcal{O}.$ Consequently, there exists $\epsilon>0$ such that $\epsilon e\in V$ and $x+\epsilon e\in x+V\subseteq U.$ By the balanced property of $V,$ for all $0<a\leq \epsilon,$ we get $ae\in V$ and hence $x+ae\in U.$ Choose a small $a'>0$ such that $x+a'e\in U.$ Then $x\in P+K-a'e$ and $x+a'e\in U.$ This implies that $x+a'e\in U\cap (P+K),$ and we arrive at a contradiction.
\end{proof}
\begin{corollary}\label{cor:lem}
 Let $K\subseteq Y$ be a nonempty proper closed convex cone and $Q\subseteq Y$ be such that $Q-K$ is closed and $Q-K\neq Y$. Let $e\in K\setminus \{0\}$. Then $$Q-K=\underset{\alpha>0}{\bigcap}(Q-K+\alpha e).$$  
\end{corollary}

Based on these, we can have the following characterization of different set order relations.

\begin{theorem}\label{new lu relation via scalarization}
     Let $P, Q\subseteq Y$  be two nonempty sets and let $K\subseteq Y$ be a nonempty proper closed convex cone. 
     \begin{itemize}
     \item[\normalfont{(i)}] Assume $P+K$ is closed. Then $P\leq_{K}^{l} Q \Longleftrightarrow \mathbb{Z}_1^{e', K} (P,Q) \leq 0~\text{for some } e'\in K\setminus \{0\} \Longleftrightarrow \mathbb{Z}_1^{e, K} (P,Q) \leq 0~\text{for all } e\in K\setminus \{0\}.$
         \item[\normalfont{(ii)}] Assume $Q-K$ is closed. Then $P\leq_{K}^{u} Q \Longleftrightarrow \mathbb{Z}_2^{e', K} (P,Q)\leq 0 ~\text{for some } e'\in K\setminus \{0\} \Longleftrightarrow \mathbb{Z}_2^{e, K} (P,Q)\leq 0 ~\text{for all } e\in K\setminus \{0\}.$ 
         
     \end{itemize}        
\end{theorem}
\begin{proof}
    ~\begin{itemize}
        \item[\normalfont{(i)}] Suppose $P\leq_{K}^{l} Q$, that is, $Q\subseteq P+K$. Then for every $q\in Q$, there exists $p_q\in P$ and $k_q\in K$ such that $q=p_q+k_q$. Thus $p_q-q=-k_q=0\cdot e-k_q\in 0\cdot e-K \text{ for all } e\in K\setminus \{0\}$. This implies $z^{e,K}(p_q-q)\leq 0 \text{ for all } e\in K\setminus \{0\}$.
        Hence $\underset{p\in P}{\inf}z^{e,K}(p-q)\leq z^{e,K}(p_q-q)\leq 0 \text{ for all } e\in K\setminus \{0\}$. This is true for every $q\in Q$. Thus $\underset{q\in Q}{\sup}~\underset{p\in P}{\inf}z^{e,K}(p-q)\leq 0$, that is, $\mathbb{Z}_1^{e, K} (P,Q) \leq 0 \text{ for all } e\in K\setminus \{0\}$.

        For the converse part, assume that $\mathbb{Z}_1^{e, K} (P,Q) \leq 0\text{ for some } e\in K\setminus \{0\}.$ Then 
        \begin{align*}
            \underset{q\in Q}{\sup}~\underset{p\in P}{\inf}z^{e,K}(p-q)\leq 0\\
            \implies \underset{p\in P}{\inf}z^{e,K}(p-q)\leq 0 \text{ for all }q\in Q.
        \end{align*}
        Choose and fix $\bar{q}\in Q$. Then 
         $\underset{p\in P}{\inf}z^{e,K}(p-\bar{q})\leq 0$. Choose $\alpha >0$. Then by the definition of infimum, there exists $\bar{p}_\alpha \in P$ such that \begin{gather*}
             z^{e,K}(\bar{p}_\alpha-\bar{q})<\alpha\\
             \implies \inf\{t\;:\; \bar{p}_\alpha-\bar{q} \in te-K\}<\alpha\\
             \implies \bar{p}_\alpha-\bar{q} \in \alpha e-K\\
             \implies \bar{p}_\alpha-\bar{q}=\alpha e-\bar{k}_q \text{ for some }\bar{k}_q\in K\\
             \implies \bar{q}=\bar{p}_\alpha+\bar{k}_q-\alpha  e\in P+K-\alpha e \;. 
         \end{gather*}
This is true for every $\alpha>0$. Thus $\bar{q}\in \underset{\alpha>0}{\bigcap}(P+K-\alpha e)=P+K$. This is in turn true for every $\bar{q}\in Q$. Thus $Q\subseteq P+K$.

\item[\normalfont{(ii)}] Suppose $P\leq_{K}^{u} Q$, that is, $P\subseteq Q-K$. Then for every $p\in P$, there exists $q_p\in Q$ and $k_p\in K$ such that $p=q_p-k_p=q_p+0\cdot e-k_p\in q_p+0\cdot e-K \text{ for all } e\in K\setminus \{0\}$. Thus $z^{e,K}(p-q_p)\leq 0$.
        This implies $\underset{q\in Q}{\inf}z^{e,K}(p-q)\leq z^{e,K}(p-q_p)\leq 0 \text{ for all } e\in K\setminus \{0\}.$ This is true for every $p\in P$. Thus $\underset{p\in P}{\sup}~\underset{q\in Q}{\inf}z^{e,K}(p-q)\leq 0$, that is, $\mathbb{Z}_2^{e, K} (P,Q) \leq 0 \text{ for all } e\in K\setminus \{0\}$.

 For the converse part, assume that $\mathbb{Z}_2^{e, K} (P,Q) \leq 0 \text{ for some } e\in K\setminus \{0\}$. Then 
        \begin{align*}
            \underset{p\in P}{\sup}~\underset{q\in Q}{\inf}z^{e,K}(p-q)\leq 0\\
            \implies \underset{q\in Q}{\inf}z^{e,K}(p-q)\leq 0 \text{ for all }p\in P.
        \end{align*}
        Choose and fix $\bar{p}\in P$. Then 
         $\underset{q\in Q}{\inf}z^{e,K}(\bar{p}-q)\leq 0$. Choose $\alpha >0$. Then by the definition of infimum, there exists $\bar{q}_\alpha \in Q$ such that \begin{gather*}
             z^{e,K}(\bar{p}-\bar{q}_\alpha)<\alpha\\
             \implies \inf\{t\;:\; \bar{p}-\bar{q}_\alpha \in te-K\}<\alpha\\
             \implies \bar{p}-\bar{q}_\alpha \in \alpha e-K\\
             \implies \bar{p}\in \bar{q}_\alpha + \alpha e-K \in Q-K+\alpha e . 
         \end{gather*}
This is true for every $\alpha>0$. Thus $\bar{p}\in \underset{\alpha>0}{\bigcap}(Q-K+\alpha e)=Q-K$. This is in turn true for every $\bar{p}\in P$. Thus $P\subseteq Q-K$.

    \end{itemize}
\end{proof}

\begin{remark}
    Our Theorem \ref{new lu relation via scalarization} improves Theorems \ref{Thm1:Kobis2016} and \ref{Thm2:Kobis2016} because we did not assume any attainment property. As mentioned earlier, in \cite{Kobis2016} the cone is not assumed to be pointed. However, the proof of Theorem \ref{new lu relation via scalarization} does not need the cone to be pointed. We also suggest a small modification in Lemma 2.2(b) from \cite{Das2024}, which is closely related to Theorem \ref{new lu relation via scalarization} (i). First let us recall the function $\mathbb{G}^{e,K}$ defined in \cite{Das2024}. For $e\in K\setminus \{0\}$, and $P,Q\in \mathcal{P}_K(Y)$, define
$$\mathbb{G}^{e,K}(P,Q)=\adjustlimits \sup_{y\in Q} \inf_{x\in P} z^{e, K}(y-x).$$ Thus $\mathbb{G}^{e,K}(P,Q)=\mathbb{Z}_2^{e,K}(Q,P)$. In \cite{Das2024} Lemma 2.2 (b), the $l$-type relation is characterized as follows: Let $P,Q\in \mathcal{P}_K(Y)$. Assume $P+K$ is closed. Then $P\leq_{K}^{l} Q \Longleftrightarrow \mathbb{G}^{e, -K} (P,Q) \leq 0$. However, the following example shows, if $e\in K\setminus \{0\}$, the above result would not hold.
\begin{example}
    Let $Y=\mathbb{R}^2,$ $K=\mathbb{R}^2_+$. Let $e=(e_1, e_2)\in K\setminus \{0\}$ and $y=(y_1, y_2)\in Y$ be any elements. We find the values of $t\in \mathbb{R}$ for which $y\in te+K.$ Observe that 
    \begin{eqnarray*}
        y\in te + K &\Longleftrightarrow & y_1- te_1\geq 0~\text{and } y_2- te_2\geq 0\\
        &\Longleftrightarrow & t \leq \begin{cases}
            \min\left\{\frac{y_1}{e_1}, \frac{y_2}{e_2}\right\} & \text{if }e_1>0,e_2>0\\
            \frac{y_1}{e_1} & \text{if }e_1>0, e_2=0\\
            \frac{y_2}{e_2} & \text{if }e_1=0, e_2>0
        \end{cases}.
    \end{eqnarray*}
    Thus $t$ can be chosen arbitrarily small and hence $\inf z^{e, -K}(y)=-\infty$. Since this is true for any $y\in Y,$ we get $\mathbb{Z}_2^{e,-K}(A, B)=-\infty$ for any $A, B\subseteq Y.$ Consider $P=\{(x,y)\in Y: (x-4)^2+(y-4)^2=1\}$ and $Q=\{(x,y)\in Y: x, y\in [-1, 1]\}.$ Then $P+K$ is closed in $Y$ and $P\nleq_{K}^{l} Q$ but $\mathbb{Z}_2^{e,-K}(P, Q)=-\infty$.
\end{example}

So in Lemma 2.2 (b) in \cite{Das2024}, $e$ should be assumed to be in $-K$. However, to avoid using $-K$ we have used $\mathbb{Z}_1^{e,K}$ in Theorem \ref{new lu relation via scalarization} (i).
\end{remark}

 A few other generalizations of extended Gerstewitz functions for set-valued optimization problems are taken from \cite{Kobis2016} as given below. 
 \begin{definition}\label{Def:scriptZ}
 Let $Y$ be a topological vector space, $K\subseteq Y$ be a nonempty proper closed convex pointed cone and $e\in K\setminus \{0\}.$
\begin{itemize}
    \item[(i)] The function $\mathscr{Z}_1^{e,K}: \mathcal{P}_K(Y) \to [-\infty, \infty]$ defined by $$\mathscr{Z}_1^{e,K}(P)= \displaystyle \inf_{x\in P} z^{e, K}(x)~~ \text{for}~ P\in \mathcal{P}_K(Y).$$
    \item[(ii)] The function $\mathscr{Z}_2^{e,K}: \mathcal{P}_K(Y) \to [-\infty, \infty]$ defined by $$\mathscr{Z}_2^{e,K}(P)= \displaystyle \sup_{x\in P} z^{e, K}(x)~~ \text{for} ~P\in \mathcal{P}_K(Y).$$
\end{itemize}   
 \end{definition}
 In Definition \ref{Def:scriptZ}, it is easy to see that if $Q=\{0\},$ then $\mathbb{Z}_i^{e,K}(P, Q)=\mathscr{Z}_i^{e,K}(P)$, where $i=1, 2.$ In paper \cite{Das2024}, a similar function $\mathcal{G}^{e,K}$ has been defined which coincides with the definition of $\mathscr{Z}_2^{e,K}$.

 The following properties of the extended Gerstewitz functions are taken from \cite{Kobis2016}, which hold even when the cone is not assumed to be pointed. But again as we showed with an example earlier, unless the pointed assumption is considered, the functions $\mathscr{Z}_1^{e, K}, \mathscr{Z}_2^{e, K}$ can be improper for every possible set $P$.  
 
\begin{lemma}\label{Lem2}
    Let $K\subseteq Y$ be a nonempty proper closed convex cone and $P, Q\in \mathcal{P}_{K}(Y)$. Then
    \begin{itemize}
        \item[\normalfont{(i)}] $P\leq_{K}^{l} Q \Longrightarrow \mathscr{Z}_1^{e,K}(P)\leq \mathscr{Z}_1^{e,K}(Q)$ for all $e\in K\setminus \{0\}.$
        
        \item[\normalfont{(ii)}]  $P\leq_{K}^{u} Q \Longrightarrow \mathscr{Z}_2^{e,K}(P)\leq \mathscr{Z}_2^{e,K}(Q)$ for all $e\in K\setminus \{0\}.$
        
        \item[\normalfont{(iii)}] For any $r\in \mathbb{R},$ $\mathscr{Z}_2^{e,K}(P)\leq r$ if and only if $P\subseteq re-K$ for all $e\in K\setminus \{0\}.$ 
       
        \item[\normalfont{(iv)}]  Let $r\in \mathbb{R}.$ Then the condition $P\subseteq re-K$ implies that $\mathscr{Z}_1^{e,K}(P)\leq r$ for all $e\in K\setminus \{0\}.$ 
    \end{itemize}
\end{lemma}
\begin{proof} ~ \begin{itemize}
    \item[\normalfont{(i)}] Corollary 2.2 of \cite{Kobis2016} gives $\mathscr{Z}_1^{e,K}(P)=\mathscr{Z}_1^{e,K}(P+K).$ As $Q\subseteq P+K,$ we get $\underset{x\in P+K}{\inf} z^{e, K}(x)\leq \underset{x\in Q}{\inf} z^{e, K}(x).$ Thus $\mathscr{Z}_1^{e,K}(P+K)\leq \mathscr{Z}_1^{e,K}(Q).$ This completes the proof.

    \item[\normalfont{(ii)}] Simply use Theorem 3.1 of \cite{Kobis2016}.
    \item[\normalfont{(iii)}] We use Theorem 2.1 $(e)$ of \cite{Kobis2016}. Fix any $r\in \mathbb{R}$ and $e\in K\setminus \{0\}.$ Assume $\mathscr{Z}_2^{e,K}(P)\leq r$. Then for any $x\in P,$ we get $\displaystyle z^{e, K}(x)\leq r.$ From the aforementioned result, $x\in re-K.$ Consequently, $P\subseteq re-K.$ Conversely, assume that  $P\subseteq re-K.$ Then for all $x'\in P,$ we get  $x'\in re-K.$ By definition, we get $\displaystyle z^{e, K}(x')\leq r.$ Consequently, $\displaystyle \sup_{x'\in P} z^{e, K}(x')\leq r.$ This completes the proof.
    \item[\normalfont{(iv)}]  Fix any $r\in \mathbb{R}$ and $e\in K\setminus \{0\}.$ Assume that  $P\subseteq re-K.$ Then for all $x'\in P,$ we get  $x'\in re-K.$ Again by Theorem 2.1 $(e)$ of \cite{Kobis2016}, we get $\displaystyle z^{e, K}(x')\leq r.$ Consequently, $\displaystyle \inf_{x'\in P} z^{e, K}(x')\leq r.$ This completes the proof.    
\end{itemize}
\end{proof}

In \cite{Das2024}, a result similar to Lemma \ref{Lem2} (i) has been given as well (see Lemma 2.3(a) in \cite{Das2024}) which states that  $P\leq_{K}^{l} Q \Longrightarrow \mathscr{Z}_2^{e,-K}(Q)\leq \mathscr{Z}_2^{e,-K}(P)$. However, if $e\in K\setminus \{0\}$, $\mathscr{Z}_2^{e,-K}(A)\in \{-\infty,\infty\}$ for all $A\subseteq Y$. So, it is not a meaningful term. It will only be meaningful if $e\in -K\setminus \{0\}$. To avoid using $-K$, we have introduced the result in terms of $\mathscr{Z}_1^{e,K}$.

The converse of Lemma \ref{Lem2}(iv) is not true in general, as we illustrate it via an example below.
\begin{example}
    Let $Y=\mathbb{R}$, $K=\mathbb{R}_+$, $P=[0,1]$ and $e=1$. Then it can be seen that $\mathscr{Z}_1^{e,K}(P)= 0$ but $P\not\subseteq 0-K=-\mathbb{R}_+$.
\end{example}

Now, we move to scalarization functions and their use in characterizing set order relations when the union of sets is considered. A few such results have been proposed in \cite{Das2024} (see Propositions 3.6, 3.8 in \cite{Das2024}). However, the results appear to be {ambiguous}. 
First, let us recall Proposition 3.6 from \cite{Das2024}.
\begin{proposition}\label{Prop3.6 of Das}
   Let $Y$ be a topological vector space, and let $K\subseteq Y$ be a nonempty proper closed convex pointed cone. Let $\{P_\gamma: \gamma \in \Gamma\}$ and $\{Q_\lambda: \lambda\in \Lambda\}$ be two collections of nonempty subsets of $Y.$ Set $P=\underset{\gamma\in \Gamma}{\bigcup} P_{\gamma}$ and $Q=\underset{\lambda\in \Lambda}{\bigcup} Q_\lambda.$ Assume that 
    \begin{itemize}
        \item[\normalfont{(i)}] for each $\lambda\in \Lambda,$ there exists $p_\lambda\in P$ such that $Q_\lambda\subseteq p_\lambda+K;$
        \item[\normalfont{(ii)}] there exists $e\in K \setminus \{0\}$ such that $\underset{\gamma\in \Gamma}{\inf}~  \mathbb{G}^{e, -K}(P_\gamma, Q_\lambda)$ is attained for all $\lambda\in \Lambda;$
        \item[\normalfont{(iii)}] $P_\gamma$ is $K$-closed for all $\gamma\in \Gamma.$ 
    \end{itemize}
    Then $P\leq_{K}^{l} Q$ if and only if $\underset{\lambda\in \Lambda}{\sup} ~\underset{\gamma\in \Gamma}{\inf}~\mathbb{G}^{e, -K}(P_\gamma, Q_\lambda)\leq 0.$
\end{proposition}
\begin{remark}\label{remark for das 3.6}
    First of all, note that if $e\in K \setminus \{0\}, \mathbb{G}^{e, -K}(A,B)\in \{-\infty,\infty\}$ for any $A,B\subseteq Y$. This is not very difficult to see. If for some $b\in B$, there do not exist any $a\in A$ and $t\in \mathbb{R}$ such that $b\in a+te+K$, then $\mathbb{G}^{e, -K}(A,B)=\infty$. On the other hand, if for each $b\in B$, there exist an $a\in A$ and $t_a\in \mathbb{R}$ such that $b\in a+t_ae+K$, then $b\in a+te+K$ for all $t\leq t_a$ and hence $\mathbb{G}^{e, -K}(A,B)=-\infty$. So $e$ must be assumed to belong to $-K\setminus \{0\}$, or one needs to use the function $\mathbb{Z}_1^{e, K}$. Even after this change, the proposition is not very meaningful, because assumption (i) in the above Proposition implies both $P\leq_{K}^{l} Q$ as well as $\underset{\lambda\in \Lambda}{\sup} ~\underset{\gamma\in \Gamma}{\inf}~\mathbb{Z}_1^{e, K}(P_\gamma, Q_\lambda)\leq 0.$ Indeed, for each $\lambda\in \Lambda,$ there exists $p_\lambda\in P$ such that $Q_\lambda\subseteq p_\lambda+K$ implies $Q=\underset{\lambda\in \Lambda}{\bigcup} Q_\lambda \subseteq \underset{\lambda \in \Lambda}{\bigcup} (p_\lambda + K) \subseteq P+K$. Also, for each $\lambda\in \Lambda,$ there exists $p_\lambda\in P$ such that $Q_\lambda\subseteq p_\lambda+K$ implies  $Q_\lambda\subseteq P_\gamma+K$, that is, $P_\gamma\leq_K^l Q_\lambda$ for all such $\gamma\in \Gamma$ such that $p_\lambda\in P_\gamma$. Hence by Theorem \ref{new lu relation via scalarization} (i), $\mathbb{Z}_1^{e, K}(P_\gamma, Q_\lambda)\leq 0$. This is true for any $\lambda\in \Lambda$ and hence $\underset{\lambda\in \Lambda}{\sup} ~\underset{\gamma\in \Gamma}{\inf}~\mathbb{Z}_1^{e, K}(P_\gamma, Q_\lambda)\leq 0.$ Thus the `if and only if' does not make any sense. We therefore propose the following modification.
\end{remark}

\begin{proposition}\label{Prop1}
    Let $Y$ be a topological vector space and let $K\subseteq Y$ be a nonempty proper closed convex pointed cone. Let $\{P_\gamma: \gamma \in \Gamma\}$ and $\{Q_\lambda: \lambda\in \Lambda\}$ be two collections of nonempty subsets of $Y.$ Set $P=\underset{\gamma\in \Gamma}{\bigcup} P_{\gamma}$ and $Q=\underset{\lambda\in \Lambda}{\bigcup} Q_\lambda.$ Assume that
    \begin{itemize}
        \item[\normalfont{(i)}] there exists $e\in K \setminus \{0\}$ such that $\underset{\gamma\in \Gamma}{\inf}~  \mathbb{Z}_1^{e, K}(P_\gamma, Q_\lambda)$ is attained for all $\lambda\in \Lambda;$
        \item[\normalfont{(ii)}] $P_\gamma$ is $K$-closed for all $\gamma\in \Gamma.$ 
    \end{itemize}
    Then $P\leq_{K}^{L} Q$ if and only if $\underset{\lambda\in \Lambda}{\sup} ~\underset{\gamma\in \Gamma}{\inf}~\mathbb{Z}_1^{e, K}(P_\gamma, Q_\lambda)\leq 0.$
\end{proposition}
 
\begin{proof}
Assume that $P\leq_{K}^{L} Q$. Let $\bar{\lambda}\in \Lambda.$ Then there exists $\bar{\gamma}\in \Gamma$ such that $P_{\bar{\gamma}}\leq_{K}^{l} Q_{\bar{\lambda}}.$ By Theorem \ref{new lu relation via scalarization}(i), we get $ \mathbb{Z}_1^{e,K}(P_{\bar{\gamma}},Q_{\bar{\lambda}})\leq 0.$ Then $$\inf_{\gamma\in \Gamma} \mathbb{Z}_1^{e,K}(P_\gamma,Q_{\bar{\lambda}})\leq \mathbb{Z}_1^{e,K}(P_{\bar{\gamma}},Q_{\bar{\lambda}})\leq 0.$$ This is true for any $\bar{\lambda}\in \Lambda.$ Thus $$\sup_{\lambda\in \Lambda} \inf_{\gamma\in \Gamma} \mathbb{Z}_1^{e, K}(P_\gamma, Q_\lambda)\leq 0.$$

Conversely, assume that $\underset{\lambda\in \Lambda}{\sup}~ \underset{\gamma\in \Gamma}{\inf} \mathbb{Z}_1^{e, K}(P_\gamma, Q_\lambda)\leq 0.$ Then $\underset{\gamma\in \Gamma}{\inf} \mathbb{Z}_1^{e, K}(P_\gamma, Q_\lambda)\leq 0$ for all $\lambda\in \Lambda.$ Let $\bar{\lambda}\in \Lambda.$ By the hypothesis, there exists $\bar{\gamma}\in \Gamma$ such that $\mathbb{Z}_1^{e, K}(P_{\bar{\gamma}}, Q_{\bar{\lambda}})\leq 0.$ By Theorem \ref{new lu relation via scalarization}(i), we get $P_{\bar{\gamma}}\leq_{K}^{l} Q_{\bar{\lambda}}.$ Consequently, $P\leq_{K}^{L} Q.$
\end{proof}

 Proposition 3.8 from \cite{Das2024} also has a similar issue. First let us recall the result.
\begin{proposition}\label{Prop3.8 of Das}
   Let $Y$ be a topological vector space and let $K\subseteq Y$ be a nonempty proper closed convex pointed cone. Let $\{P_\gamma: \gamma \in \Gamma\}$ and $\{Q_\lambda: \lambda\in \Lambda\}$ be two collections of nonempty subsets of $Y.$ Set $P=\underset{\gamma\in \Gamma}{\bigcup} P_{\gamma}$ and $Q=\underset{\lambda\in \Lambda}{\bigcup} Q_\lambda.$ Assume that 
    \begin{itemize}
        \item[\normalfont{(i)}] for each $\gamma\in \Gamma,$ there exists $q_\gamma\in Q$ such that $P_\gamma\subseteq q_\gamma-K;$
        \item[\normalfont{(ii)}] there exists $e\in K \setminus \{0\}$ such that $\underset{\lambda\in \Lambda}{\inf}~  \mathbb{Z}_2^{e, K}(P_\gamma, Q_\lambda)$ is attained for all $\gamma\in \Gamma;$
        \item[\normalfont{(iii)}] $Q_\lambda$ is $-K$-closed for all $\lambda\in \Lambda.$ 
    \end{itemize}
    Then $P\leq_{K}^{u} Q$ if and only if $\underset{\gamma\in \Gamma}{\sup} ~\underset{\lambda\in \Lambda}{\inf}~\mathbb{Z}_2^{e, K}(P_\gamma, Q_\lambda)\leq 0.$
\end{proposition}
\begin{remark}
    Here again, assumption (i) in the above Proposition implies both $P\leq_{K}^{u} Q$ as well as  $\underset{\gamma\in \Gamma}{\sup} ~\underset{\lambda\in \Lambda}{\inf}~\mathbb{Z}_2^{e, K}(P_\gamma, Q_\lambda)\leq 0$ and hence the `if and only if' statement does not make any sense. We therefore propose the following modification.  
\end{remark}

\begin{proposition}\label{Prop2}
    Let $Y$ be a topological vector space and let $K\subseteq Y$ be a nonempty proper closed convex pointed cone. Let $\{P_\gamma: \gamma \in \Gamma\}$ and $\{Q_\lambda: \lambda\in \Lambda\}$ be two collections of nonempty subsets of $Y.$ Set $P=\underset{\gamma\in \Gamma}{\bigcup} P_{\gamma}$ and $Q=\underset{\lambda\in \Lambda}{\bigcup} Q_\lambda.$ Assume that
    \begin{itemize}
        \item[\normalfont{(i)}] there exists $e\in K\setminus \{0\}$ such that $\underset{\lambda\in \Lambda}{\inf} \mathbb{Z}_2^{e, K}(P_\gamma, Q_\lambda)$ is attained for all $\gamma\in \Gamma;$
        \item[\normalfont{(ii)}] $Q_\lambda$ is $-K$-closed for all $\lambda\in \Lambda.$ 
    \end{itemize}
    Then $P\leq_{K}^{U} Q$ if and only if $\underset{\gamma\in \Gamma}{\sup}~ \underset{\lambda\in \Lambda}{\inf} \mathbb{Z}_2^{e, K}(P_\gamma, Q_\lambda)\leq 0.$
\end{proposition}
\begin{proof}
Assume that $P\leq_{K}^{U} Q$. Let $\bar{\gamma}\in \Gamma.$ Then there exists $\bar{\lambda}\in \Lambda$ such that $P_{\bar{\gamma}}\leq_{K}^{u} Q_{\bar{\lambda}}.$ By Theorem \ref{new lu relation via scalarization}(ii), we get $ \mathbb{Z}_2^{e,K}(P_{\bar{\gamma}}, Q_{\bar{\lambda}})\leq 0.$ Then $$\inf_{\lambda\in \Lambda} \mathbb{Z}_2^{e,K}(P_{\bar{\gamma}}, Q_{\lambda})\leq \mathbb{Z}_2^{e,K}(P_{\bar{\gamma}}, Q_{\bar{\lambda}})\leq 0.$$ Thus $$\sup_{\gamma\in \Gamma} \inf_{\lambda\in \Lambda} \mathbb{Z}_2^{e, K}(P_\gamma, Q_\lambda)\leq 0.$$

Conversely, assume that \begin{align*}
    \sup_{\gamma\in \Gamma} \inf_{\lambda\in \Lambda} \mathbb{Z}_2^{e, K}(P_\gamma, Q_\lambda)\leq 0.
\end{align*} Then $\underset{\lambda\in \Lambda} {\inf}\mathbb{Z}_2^{e, K}(P_\gamma, Q_\lambda)\leq 0$ for all $\gamma\in \Gamma.$ Let $\bar{\gamma}\in \Gamma.$ By the hypothesis, there exists $\bar{\lambda}\in \Lambda$ such that $\mathbb{Z}_2^{e, K}(P_{\bar{\gamma}}, Q_{\bar{\lambda}})\leq 0.$ By Theorem \ref{new lu relation via scalarization}(ii), we get $P_{\bar{\gamma}}\leq_{K}^{u} Q_{\bar{\lambda}}.$ Consequently, $P\leq_{K}^{U} Q.$
\end{proof}

\begin{remark}
    Since Proposition 3.9 in \cite{Das2024} is a combination of Propositions 3.6 and 3.8 of the same paper, this also suffers from the same ambiguity. Thus we propose the following modification to Proposition 3.9 in \cite{Das2024}, which is easy to prove by combining Proposition \ref{Prop1} and Proposition \ref{Prop2}.

\end{remark}

\begin{proposition}\label{Prop3}
   Let $Y$ be a topological vector space and let $K\subseteq Y$ be a nonempty proper closed convex cone. Let $\{P_\gamma: \gamma \in \Gamma\}$ and $\{Q_\lambda: \lambda\in \Lambda\}$ be the collections of nonempty subsets of $Y.$ Set $P=\underset{\gamma\in \Gamma}{\bigcup} P_{\gamma}$ and $Q=\underset{\lambda\in \Lambda} {\bigcup}Q_\lambda.$ Assume that
    \begin{itemize}
        \item[\normalfont{(i)}] there exist $e_1, e_2\in K\setminus \{0\}$ such that $ \underset{\gamma\in \Gamma}{\inf}~ \mathbb{Z}_1^{e_1, K}(P_\gamma, Q_\lambda)$ is attained for all $\lambda\in \Lambda$;
        $\underset{\lambda\in \Lambda}{\inf}~ \mathbb{Z}_2^{e_2, K}(P_\gamma, Q_\lambda)$ is attained for all $\gamma\in \Gamma$ and 
        \item[\normalfont{(ii)}] $P_\gamma$ and $Q_\lambda$ are $K$-closed and $-K$-closed, respectively, for all $\gamma\in \Gamma$ and for all $\lambda\in \Lambda.$ 
    \end{itemize}
    Then $P\leq_{K}^{S} Q$ if and only if $\underset{\lambda\in \Lambda}{\sup} \underset{\gamma\in \Gamma}{\inf} \mathbb{Z}_1^{e_1, K}(P_\gamma, Q_\lambda)\leq 0$ and $\underset{\gamma\in \Gamma}{\sup} \underset{\lambda\in \Lambda}{\inf} \mathbb{Z}_2^{e_2, K}(P_\gamma, Q_\lambda)\leq 0.$
\end{proposition}

We also noticed inconsistencies in Propositions 3.5 and 3.7 of \cite{Das2024} as well. Recall Proposition 3.5 from \cite{Das2024}.

\begin{proposition}\label{Prop3.5 of Das}
      Let $Y$ be a topological vector space and let $K\subseteq Y$ be a nonempty proper closed convex pointed cone. Let $\{P_\gamma: \gamma \in \Gamma\}$ and $\{Q_\lambda: \lambda\in \Lambda\}$ be two collections of nonempty subsets of $Y.$ Set $P=\underset{\gamma\in \Gamma}{\bigcup} P_{\gamma}$ and $Q=\underset{\lambda\in \Lambda} {\bigcup}Q_\lambda.$ Assume that for each $\lambda\in \Lambda,$ there exists $p_\lambda\in P$ such that $Q_\lambda \subseteq p_\lambda +K.$ If $P\leq_{K}^{l} Q,$ then 
      $$\underset{\lambda\in \Lambda}{\sup}~\mathcal{G}^{e, -K}(Q_\lambda)\leq \underset{\gamma\in \Gamma}{\sup}~\mathcal{G}^{e, -K}(P_\gamma)~\text{for all } e\in K\setminus \{0\}.$$
\end{proposition}
\begin{remark}
    First of all, as we mentioned earlier, if $e\in K\setminus \{0\}$, $\mathcal{G}^{e, -K}(A)\in \{-\infty,\infty\}$ for all $A\subseteq Y$. So either $e$ must be assumed to be in $-K\setminus \{0\}$ or one should use $\mathscr{Z}_1^{e, K}$. Secondly, the assumption that for each $\lambda\in \Lambda,$ there exists $p_\lambda\in P$ such that $Q_\lambda \subseteq p_\lambda +K$ implies $P\leq_{K}^{l} Q$. So the `if' statement in the conclusion is a result of the assumption, and the `if-then' in the conclusion is meaningless. We propose the following modification, where we drop this assumption.  
\end{remark}

\begin{proposition}\label{L and mathscr Z}
      Let $Y$ be a topological vector space, and let $K\subseteq Y$ be a nonempty proper closed convex pointed cone. Let $\{P_\gamma: \gamma \in \Gamma\}$ and $\{Q_\lambda: \lambda\in \Lambda\}$ be two collections of nonempty subsets of $Y.$ Set $P=\underset{\gamma\in \Gamma}{\bigcup} P_{\gamma}$ and $Q=\underset{\lambda\in \Lambda} {\bigcup}Q_\lambda.$  If $P\leq_{K}^{L} Q,$ then 
      $$\underset{\gamma\in \Gamma}{\inf}~\mathscr{Z}_1^{e, K}(P_\gamma)\leq \underset{\lambda\in \Lambda}{\inf}~\mathscr{Z}_1^{e, K}(Q_\lambda)~\text{for all } e\in K\setminus \{0\}.$$
\end{proposition}
\begin{proof}
    Assume that $P\leq_{K}^{L} Q.$ Then for $\lambda\in \Lambda,$ there exists $\gamma\in \Gamma$ such that $P_\gamma\leq_{K}^{l} Q_\lambda.$ By Lemma \ref{Lem2}(i), we get $\mathscr{Z}_1^{e, K}(P_\gamma)\leq \mathscr{Z}_1^{e, K}(Q_\lambda).$ Consequently, $\underset{\gamma\in \Gamma}{\inf}~\mathscr{Z}_1^{e, K}(P_\gamma)\leq \underset{\lambda\in \Lambda}{\inf}~\mathscr{Z}_1^{e, K}(Q_\lambda)~\text{for all } e\in K\setminus \{0\}.$
\end{proof}

A similar modification is required for Proposition 3.7 in \cite{Das2024} as well. First, recall Proposition 3.7 from \cite{Das2024}.

\begin{proposition}\label{Prop3.7 of Das}
   Let $Y$ be a topological vector space and let $K\subseteq Y$ be a nonempty proper closed convex pointed cone. Let $\{P_\gamma: \gamma \in \Gamma\}$ and $\{Q_\lambda: \lambda\in \Lambda\}$ be two collections of nonempty subsets of $Y.$ Set $P=\underset{\gamma\in \Gamma}{\bigcup} P_{\gamma}$ and $Q=\underset{\lambda\in \Lambda} {\bigcup}Q_\lambda.$ Assume that for each $\gamma\in \Gamma,$ there exists $q_\gamma\in Q$ such that $P_\gamma \subseteq q_\gamma -K.$ If $P\leq_{K}^{u} Q,$ then 
      $$\underset{\gamma\in \Gamma}{\sup}~\mathcal{G}^{e, K}(P_\gamma)\leq \underset{\lambda\in \Lambda}{\sup}~\mathcal{G}^{e, K}(Q_\lambda)~\text{for all } e\in K\setminus \{0\}.$$ 
\end{proposition}
Here again, the assumption that for each $\gamma\in \Gamma,$ there exists $q_\gamma\in Q$ such that $P_\gamma \subseteq q_\gamma -K$ implies the `if' part of the conclusion. So we modify Proposition \ref{Prop3.7 of Das} by omitting this assumption. Its proof is very similar to that of Proposition \ref{L and mathscr Z} and hence omitted. 
\begin{proposition}\label{Mod:Prop3.7}
      Let $Y$ be a topological vector space, and let $K\subseteq Y$ be a nonempty proper closed convex pointed cone. Let $\{P_\gamma: \gamma \in \Gamma\}$ and $\{Q_\lambda: \lambda\in \Lambda\}$ be two collections of nonempty subsets of $Y.$ Set $P=\underset{\gamma\in \Gamma}{\bigcup} P_{\gamma}$ and $Q=\underset{\lambda\in \Lambda} {\bigcup}Q_\lambda.$  If $P\leq_{K}^{U} Q,$ then 
      $$\underset{\gamma\in \Gamma}{\sup}~\mathscr{Z}_2^{e, K}(P_\gamma)\leq \underset{\lambda\in \Lambda}{\sup}~\mathscr{Z}_2^{e, K}(Q_\lambda)~\text{for all } e\in K\setminus \{0\}.$$
\end{proposition}

Some sufficiency criteria of $l$-type and $u$-type set order relations via scalarization have been given in \cite{Kobis2016} (see Theorems 3.6 and 3.9 in \cite{Kobis2016}), where the dual cone of $K$ has been used. We recall them here as they will be used to study robust solutions to an uncertain set-valued optimization problem. Let $Y$ be a locally convex topological vector space and $K\subseteq Y$ be a proper closed convex pointed cone in $Y.$ Let $Y^*$ be the dual of $Y$ consisting of all continuous linear functionals. The dual cone of $K$ is denoted by $K^*$, that is $K^*=\{w\in Y^*\;:\;w(k)\geq 0 \;\text{for all }k\in K\}$. Consider the half space $K_w$ generated by some $w\in K^*\setminus \{0\}$, that is, $$K_w=\{y\in Y: w(y)\geq 0\}.$$

We now recall Theorems 3.6 and 3.9 from \cite{Kobis2016}.

\begin{theorem}\label{Kobis Thm 3.6}
    Let $Y$ be a locally convex topological space, $P, Q\subseteq Y$ be two nonempty subsets and $K$ be a proper closed convex cone in $Y.$ Suppose that for any $w\in K^*\setminus \{0\},$ there exists $e_w\in \operatorname{int}(K_w)$ such that $\mathscr{Z}_2^{e_w, K_w}(P)\leq \mathscr{Z}_2^{e_w, K_w}(Q).$ If $Q-K$ is closed and convex, then $P\leq_{K}^{u} Q.$
\end{theorem}

\begin{theorem}\label{Kobis Thm 3.9}
    Let $Y$ be a locally convex topological space, $P, Q\subseteq Y$ be two nonempty subsets and $K$ be a proper closed convex cone in $Y.$ Suppose that for any $w\in K^*\setminus \{0\},$ there exists $e_w\in $ $K_w\setminus \{0\}$ such that $\mathscr{Z}_1^{e_w, K_w}(P)\leq \mathscr{Z}_1^{e_w, K_w}(Q).$ If $P+K$ is closed and convex, then $P\leq_{K}^{l} Q.$
\end{theorem}

As mentioned in \cite{Kobis2016}, it should be noted that in the above theorems, $K_w$ need not be a pointed cone even though $K$ is a pointed cone. We now proceed to study uncertain set-valued optimization problem in the next section.

\section{Uncertain Set-Valued Optimization Problem (USOP)}\label{uncertain set valued problem}
Let $X$ be a linear space. Let $Y$ and $Z$ be the topological vector spaces partially ordered by nonempty closed convex pointed cone $K\subseteq Y$ and $K'\subseteq Z$, respectively. Let $\mathscr{X}$ be a subset of $X$, and let $\mathscr{U}\subseteq \mathbb{R}^n$ be the uncertainty set, which is assumed to be nonempty and compact. The paper \cite{Das2024} considers the following USOP with objective function $H:{X}\times \mathscr{U}\to \mathcal{P}(Y)$ and the constraint set given by $\{F_i: {X}\times \mathscr{U}\to \mathcal{P}(Z): 1\leq i\leq m\}$ for some $m\in \mathbb{N}:$
\begin{equation} \label{SP(u)}
\begin{aligned}
    & 
    \begin{cases}
       \underset{x}{\text{minimize }}  H(x,u) \\
        \text{subject to: } x \in \mathscr{X},\\
        F_i(x, u) \subseteq -K', \;\; i = 1, \cdots, m;
    \end{cases}
\end{aligned}
\end{equation}
where $u\in \mathscr{U}$ is an uncertainty parameter. 
This problem in itself is not well defined, and there are multiple ways to interpret the Problem (\ref{SP(u)}). For example, for each $u\in \mathscr{U}$ one can define the problem:
\begin{equation*} \label{SP(u1)}
\text{(SP($u$))} \quad
\begin{aligned}
    & 
    \begin{cases}
       \underset{x}{\text{minimize }}  H(x,u) \\
        \text{subject to: } x \in \mathscr{X}~\text{ and }\\ 
        F_1(x, u) \subseteq -K', \cdots,  F_m(x, u) \subseteq -K'.
    \end{cases}
\end{aligned}
\end{equation*}

i.e., for each $u\in \mathscr{U},$ (SP($u$)) is a constrained set-valued optimization problem and hence Problem (\ref{SP(u)}) can be thought of as a parametrized family of set-valued optimization problems, parametrized by $u\in \mathscr{U}$. 

However, the most interesting reformulation is the worst-case reformulation or robust reformulation that considers the problem \begin{equation*} \label{SP}
\begin{aligned}
    & 
    \begin{cases}
       \underset{x}{\text{minimize }} \underset{u \in \mathscr{U}}{\sup} H(x,u) \\
        \text{subject to: }  x\in \mathscr{X},\\
        F_i(x, u) \subseteq -K'~\text{ for all } u\in \mathscr{U}~\text{ and }  i = 1, \cdots, m.
    \end{cases}
\end{aligned}
\end{equation*}
However, in set-valued optimization, one needs to understand, how to interpret the supremum term in the above reformulation. Usually, it is given in terms of the following set-valued map $H_{\mathscr{U}}: \mathscr{X}\to \mathcal{P}(Y)$, defined as $$H_{\mathscr{U}}(x)=\bigcup_{u\in \mathscr{U}} H(x,u)~~~\text{for all }  x\in \mathscr{X}.$$

Thus, the worst-case robust reformulation of Problem (\ref{SP(u)}) is 
 \begin{equation*} \label{SP1}
\text{$(SP(\mathscr{U}))_{RC}$} \quad
    \begin{cases}
        \min H_{\mathscr{U}}(x) \\
        \text{Subject to: }   \displaystyle \bigcup_{u\in \mathscr{U}} F_i(x, u) \subseteq -K' ~\text{ for all }~i=1, 2, \cdots, m.
    \end{cases}
\end{equation*}

By applying Lemma \ref{Lem2}(iii) to the constraints of $(SP(\mathscr{U}))_{RC}$, we get $ \mathscr{Z}_2^{e', K'}(F_i(x,u))\leq 0$  for any $e'\in K'\setminus \{0\}$ and for all $u\in \mathscr{U}, i=1, 2, \cdots, m$. Consequently, $$ \sup_{u\in \mathscr{U}}\mathscr{Z}_2^{e', K'}(F_i(x,u))\leq 0~~\text{for all }~e'\in K'\setminus \{0\}, i=1, 2, \cdots, m.$$

We could have taken $\mathscr{Z}_1$ instead of $\mathscr{Z}_2$ here as well. However, in view of Lemma \ref{Lem2}(iv), we see that we get an `if and only if' criteria using $\mathscr{Z}_2$ and not $\mathscr{Z}_1$. Thus, we chose $\mathscr{Z}_2$ here. Let $e'\in K'\setminus \{0\}.$ Define $\mathscr{F}: \mathscr{X}\to \mathbb{R}^m$ by $$\mathscr{F}(x)= \left (\sup_{u\in \mathscr{U}}\mathscr{Z}_2^{e', K'}(F_1(x,u) ), \cdots, \sup_{u\in \mathscr{U}}\mathscr{Z}_2^{e', K'}(F_m(x,u))\right)~~ \text{for all}~~ x\in \mathscr{X}.$$

Denote the feasible set of $(SP(\mathscr{U}))_{RC}$ by $\mathcal{S}=\{x\in \mathscr{X}: \mathscr{F}(x)\in \mathscr{M}=-\mathbb{R}^m_+\}$ which is called the robust feasible set of (\ref{SP(u)}). Note that, though $\mathscr{F}$ has a dependency on $e'\in K'\setminus \{0\}$, the feasible set $\mathcal{S}$ is independent of the choice of $e'$. Thus, whenever we mention about $\mathcal{S}$, we would not explicitly mention about $e'$. The following notions of robust solution for problem (\ref{SP(u)}) have been defined in \cite{Das2024}.
\begin{definition}\cite{Das2024}
    An element $x^*\in \mathcal{S}$ is said to be a $\leq_{K}^{v}$-robust solution for (\ref{SP(u)}) if there exists no $x_0\in \mathcal{S}\setminus \{x^*\}$ such that $H_{\mathscr{U}}(x_0)\leq_{K}^{v} H_{\mathscr{U}}(x^*),$ where $v\in \{l, u, s\}.$
\end{definition}

\begin{definition}
    An element $x^*\in \mathcal{S}$ is said to be a $\leq_{K}^{v}$-robust solution for (\ref{SP(u)}) if for every $x_0\in \mathcal{S}\setminus \{x^*\}$ whenever $H_{\mathscr{U}}(x_0)\leq_{K}^{v} H_{\mathscr{U}}(x^*),$, one have $H_{\mathscr{U}}(x^*)\leq_{K}^{v} H_{\mathscr{U}}(x_0),$ where $v\in \{l, u, s\}.$
\end{definition}

However, looking at various results correspondoing to robust solutions given in \cite{Das2024}, we are motivated to define robust solutions through the help of the set order relations $\leq_{K}^V$, where $V$ can be one of $\{L, U, S\}$.
\begin{definition}
    An element $x^*\in \mathcal{S}$ is said to be a $\leq_{K}^{V}$-robust solution for (\ref{SP(u)}) if for every $x_0\in \mathcal{S}\setminus \{x^*\}$ whenever $H_{\mathscr{U}}(x_0)\leq_{K}^{v} H_{\mathscr{U}}(x^*),$, one have $H_{\mathscr{U}}(x^*)\leq_{K}^{V} H_{\mathscr{U}}(x_0),$ where $V\in \{L, U, S\}.$
\end{definition}

From Proposition \ref{Ll relations}, we can see that every $\leq_K^l$-robust (similarly $\leq_K^u$-robust, $\leq_K^s$-robust, respectively) solution of (\ref{SP(u)}) is a  $\leq_K^L$-robust (similarly $\leq_K^U$-robust, $\leq_K^S$-robust, respectively) solution. So, our introduced notions are indeed weaker. The paper \cite{Das2024} characterizes various robust solutions via scalarization functions. However, we point out some inconsistencies and propose modifications to some of the results.

\subsection{Characterization of Robust Solution for USOP}\label{Characterization of Robust Solution for USOP}

First we need to recall some of the notations used in \cite{Das2024}. Let $x^*\in \mathscr{X}$. Define $B^1_{x^*}:\mathscr{X}\to \mathbb{R}^{1+m}$ by 
$$B^1_{x^*}(x)=\left(\underset{a\in \mathscr{U}}{\sup }\mathcal{G}^{e, -K }(H(x^*, a))-  \underset{u\in \mathscr{U}}{\sup }\mathcal{G}^{e, -K }(H(x, u)), \mathscr{F}(x)\right)$$ and consider the sets $R^1_{x^*}(x)=\left\{B^1_{x^*}(x): x\in \mathscr{X}\setminus\{x^*\}\right\}$ and $C=\{(x,y)\in \mathbb{R}^{1+m}: u\leq 0, v\in \mathscr{M}=-\mathbb{R}^2_+\}.$ In \cite{Das2024} Proposition 4.1, the authors used the set $R^1_{x^*}$ to provide a necessary condition for a $\leq_{K}^{l}$-robust solution of (\ref{SP(u)}). We recall Proposition 4.1 from \cite{Das2024}.

\begin{proposition}\label{Prop4.1 of Das}
    Consider problem (\ref{SP(u)}). Let $x^*\in \mathcal{S}.$ Assume that for any $u\in \mathscr{U},$ there exists $p_u\in H_{\mathscr{U}}(x)$ for all $x\in \mathscr{X}\setminus\{x^*\}$ such that $H(x^*, u)\subseteq p_u+K.$ If $x^*$ is a $\leq_{K}^{l}$ robust solution, then $R^1_{x^*}\cap C=\emptyset,$ or equivalently, the generalized system $B^1_{x^*}(x)\in C,$ $x\in \mathcal{S}\setminus\{x^*\}$ is inconsistent. 
\end{proposition}

\begin{remark}
    First of all, the definition $B^1_{x^*}(x)$ would not be meaningful if $e\in K\setminus \{0\}$, as we showed earlier that $\mathcal{G}^{e, -K }$ becomes improper in that case. Secondly, the condition that for any $u\in \mathscr{U},$ there exists $p_u\in H_{\mathscr{U}}(x)$ for all $x\in \mathscr{X}\setminus\{x^*\}$ such that $H(x^*, u)\subseteq p_u+K$ implies $ H_{\mathscr{U}}(x) \leq_K^l  H_{\mathscr{U}}(x^*)$ (as well as $H_{\mathscr{U}}(x) \leq_K^L  H_{\mathscr{U}}(x^*)$). Thus $x^*$ can not be a $\leq_{K}^{l}$ robust solution and the conclusion of the proposition is not meaningful. Therefore modification in the result is needed, and we propose one such modification below.
\end{remark}

First, in parallel to $B^1_{x^*}, \;R^1_{x^*}$, for $x^*\in \mathscr{X}$ and $e\in K\setminus \{0\}$, we define   
\begin{gather*}
    \hat{B}^1_{x^*,e}(x)=\left(\underset{u\in \mathscr{U}}{\inf }\mathscr{Z}_1^{e, K }(H(x, u))-  \underset{a\in \mathscr{U}}{\inf }\mathscr{Z}_1^{e, K }(H(x^*, a)), \mathscr{F}(x)\right),\\ \text{and } \hat{R}^1_{x^*,e}(x)=\left\{\hat{B}^1_{x^*,e}(x): x\in \mathscr{X}\setminus\{x^*\}\right\}.
\end{gather*} 

We kept $e$ in the notations to stress its dependency on $e$. We only consider those situations where $\underset{u\in \mathscr{U}}{\sup }\mathscr{Z}_1^{e, K }(H(x, u))$ is finite for all $x\in \mathscr{X}$, so that $\hat{B}^1_{x^*,e}(x)$ is meaningful. With the help of $\hat{B}^1_{x^*,e}(x)$ and $\hat{R}^1_{x^*,e}$, we can have the following modification to the above proposition. 
\begin{proposition}\label{Prop4.1 of Das renewed}
    Consider problem (\ref{SP(u)}). Let $x^*\in \mathcal{S},$ and $e\in K\setminus \{0\}$. If  $\hat{R}^1_{x^*,e}\cap C=\emptyset$ (equivalently $\hat{B}^1_{x^*,e}(x)\in C,$ $x\in \mathcal{S}\setminus\{x^*\}$ is inconsistent), then $x^*$ is a $\leq_{K}^{L}$ robust solution. 
\end{proposition}
\begin{proof}
    Let $\hat{R}^1_{x^*,e}\cap C=\emptyset$. If possible, assume that $x^*$ is not a $\leq_{K}^{L}$ robust solution. Then there exists $x\in \mathcal{S}\setminus\{x^*\}$ such that $H_{\mathscr{U}}(x)\leq_K^L H_{\mathscr{U}}(x^*)$, i.e., $\underset{u\in \mathscr{U}}{\bigcup} H(x,u)\leq_K^L \underset{a\in \mathscr{U}}{\bigcup} H(x^*,a)$. But then, using Proposition \ref{L and mathscr Z}, we get $\underset{u\in \mathscr{U}}{\inf }\mathscr{Z}_1^{e, K }(H(x, u)) \leq \underset{a\in \mathscr{U}}{\inf }\mathscr{Z}_1^{e, K }(H(x^*, a))$. Thus $\hat{B}^1_{x^*,e}(x)\in \hat{R}^1_{x^*,e}\cap C$, contradicting the assumption. Hence $x^*$ must be a $\leq_{K}^{L}$ robust solution.
\end{proof}

It can be seen that Proposition \ref{Prop4.1 of Das renewed} provides a sufficiency criteria, whereas  Proposition \ref{Prop4.1 of Das} provides a necessary condition. However, keeping Proposition \ref{L and mathscr Z} in mind, getting a necessary condition will  be difficult. We will show a few such necessary conditions for being a robust solution later. 

The paper \cite{Das2024} also characterizes $\leq_{K}^{v}$ robust solution for (\ref{SP(u)}), where $v\in \{l, u, s\}$ via different other set functions. However, all these results have inconsistencies. Thus, we propose various modifications. We first recall all those functions used in \cite{Das2024} here in a combined manner. Let $x^*\in \mathscr{X}$ and $e\in K\setminus \{0\}$. Define the map $B^i_{x^*}: \mathscr{X}\to \mathbb{R}^{1+m},$ $i=2, 3, 4$ by 
\begin{eqnarray*}
    B^2_{x^*}(x)&=&\left(\underset{a\in \mathscr{U}}{\sup } \underset{u\in \mathscr{U}}{\inf } \mathbb{G}^{e, -K} (H(x,u), H(x^*, a)), \mathscr{F}(x)\right);\\
    B^3_{x^*}(x)&=&\left(\underset{u\in \mathscr{U}}{\sup } \mathcal{G}^{e, K} (H(x,u))- \underset{a\in \mathcal{U}}{\sup } ~\mathcal{G}^{e, K}(H(x^*, a)), \mathscr{F}(x)\right);\\
    B^4_{x^*}(x)&=&\left(\underset{u\in \mathscr{U}}{\sup } \underset{a\in \mathscr{U}}{\inf } \mathbb{G}^{e, K} (H(x^*, a), H(x,u)), \mathscr{F}(x)\right);\\
    B^5_{x^*}(x)&=&\Big(\max\Big\{\underset{a\in \mathscr{U}}{\sup } \underset{u\in \mathscr{U}}{\inf } \mathbb{G}^{e, -K} (H(x,u), H(x^*, a)),\\ && \underset{u\in \mathscr{U}}{\sup } \underset{a\in \mathscr{U}}{\inf } \mathbb{G}^{e, K} (H(x^*, a), H(x,u))\Big\}, \mathscr{F}(x)\Big).
\end{eqnarray*}
and consider the sets 
\begin{eqnarray*}
   R^i_{x^*}&=&\left\{B^i_{x^*}(x): x\in \mathscr{X}\setminus\{x^*\}\right\};\;i=2,3,4,5.
\end{eqnarray*}
Since some of these definitions need modifications (like $e$ needs to be assumed in $-K\setminus \{0\}$, etc.), in parallel to these functions and sets, we define the following:

For $x^*\in  \mathscr{X}$ and $e\in K\setminus \{0\}$, define
\begin{eqnarray*}
    \hat{B}^2_{x^*,e}(x)&=&\left(\underset{a\in \mathscr{U}}{\sup } \underset{u\in \mathscr{U}}{\inf } \mathbb{Z}_1^{e, K} (H(x,u), H(x^*, a)), \mathscr{F}(x)\right);\\
    \hat{B}^3_{x^*,e}(x)&=&\left(\underset{u\in \mathscr{U}}{\sup } \mathscr{Z}_2^{e, K} (H(x,u))- \underset{a\in \mathscr{U}}{\sup } ~\mathscr{Z}_2^{e, K}(H(x^*, a)), \mathscr{F}(x)\right);\\
    \hat{B}^4_{x^*,e}(x)&=&\left(\underset{u\in \mathscr{U}}{\sup } \underset{a\in \mathscr{U}}{\inf } \mathbb{Z}_2^{e, K} (H(x,u), H(x^*, a)), \mathscr{F}(x)\right);\\
    \hat{B}^5_{x^*,e}(x)&=&\Big(\max\Big\{\underset{a\in \mathscr{U}}{\sup } \underset{u\in \mathscr{U}}{\inf } \mathbb{Z}_1^{e, K} (H(x,u), H(x^*, a)),\\ && \underset{u\in \mathscr{U}}{\sup } \underset{a\in \mathscr{U}}{\inf } \mathbb{Z}_2^{e, K} (H(x,u), H(x^*, a))\Big\}, \mathscr{F}(x)\Big).
\end{eqnarray*}
and consider the sets 
\begin{eqnarray*}
  \hat{R}^i_{x^*,e}&=&\left\{\hat{B}^i_{x^*, e}(x): x\in \mathscr{X}\setminus\{x^*\}\right\};\;i=2,3,4,5.
\end{eqnarray*}
It should be noted that $B^3_{x^*}=\hat{B}^3_{x^*, e}$, but to unify the notation, we reintroduced it. We now recall various results given in  \cite{Das2024}, which require modifications. We start by recalling Theorem 4.1 from \cite{Das2024}.
\begin{theorem}\label{Thm4.1 of Das}
    Consider problem (\ref{SP(u)}). Let $x^*\in \mathcal{S}.$ Assume that 
    \begin{itemize}
        \item[\normalfont{(i)}] for any $u\in \mathscr{U}$ and $x\in \mathscr{X}\setminus \{x^*\},$ there exists an element $p_u\in H_{\mathscr{U}}(x)$ such that $H(x^*, u)\subseteq p_u+K;$
        \item[\normalfont{(ii)}] there exists $e\in K\setminus \{0\}$ such that the infimum $\underset{u\in \mathscr{U}}{\inf} \mathbb{G}^{e, -K} (H(x,u), H(x^*, a))$ is attained for all $a\in \mathscr{U};$ and 
        \item[\normalfont{(iii)}] $H(x,u)$ is $K$-closed for all $x$ and $u.$
    \end{itemize}
    Then $x^*$ is a $\leq_{K}^{l}$-robust solution for (\ref{SP(u)}) if and only if $R^2_{x^*}\cap C=\emptyset,$ or equivalently, the generalized system $B^2_{x^*}(x)\in C$ is inconsistent. 
\end{theorem}
\begin{remark}
    Here again, assumption (i) implies that $ H_{\mathscr{U}}(x) \leq_K^l  H_{\mathscr{U}}(x^*)$ for all $x\in  \mathcal{S}\setminus\{x^*\}$. Thus $x^*$ cannot be $\leq_{K}^{l}$-robust solution for (\ref{SP(u)}). Also, $e$ must be assumed to belong to $-K\setminus \{0\}$. Therefore we propose the following modification of Theorem \ref{Thm4.1 of Das} below.
\end{remark}

\begin{theorem}\label{Main2}
    Consider problem (\ref{SP(u)}). Let $x^*\in \mathcal{S}.$ Assume that
    \begin{itemize}
        \item[\normalfont{(i)}]  $H(x, u)+K$ is closed for every $x$ in $\mathscr{X}$ and $u$ in $\mathscr{U}$;
        \item[\normalfont{(ii)}]  there exists $e\in  K\setminus \{0\}$ such that $\underset{u\in \mathscr{U}}{\inf}~ \mathbb{Z}_1^{e,K}(H(x,u), H(x^*, a))$ is attained for all $x \in \mathscr{X}\setminus\{x^*\}$ and $a \in \mathscr{U}.$
    \end{itemize}
    Then $x^*$ is a $\leq_{K}^{L}$-robust solution for (\ref{SP(u)}) if and only if $\hat{R}^2_{x^*,e}\cap C=\emptyset.$
\end{theorem}
\begin{proof}
    Let $x^*\in \mathcal{S}$ be a $\leq_{K}^{L}$-robust solution. 
    On the contrary assume that $\hat{R}^2_{x^*, e}\cap C\neq \emptyset$. Let $\hat{v} \in \hat{R}^2_{x^*, e}\cap C.$  Then $\hat{v}= \hat{B}_{x^*, e}^2(\hat{x})$ for some $\hat{x}\in \mathscr{X}\setminus \{x^*\}$. Since $\hat{v}\in C$, $\hat{x}\in \mathcal{S}\setminus \{x^*\}$ and $\underset{a\in \mathscr{U}}{\sup}~\underset{u\in \mathscr{U}}{\inf } \mathbb{Z}_1^{e,K}(H(\hat{x},u), H(x^*, a))\leq 0.$ Then for every $a\in \mathscr{U},$ $\underset{u\in \mathscr{U}}{\inf } \mathbb{Z}_1^{e,K}(H(\hat{x},u), H(x^*, a))\leq 0$ and by the hypothesis (ii), this infimum is attained. By Proposition \ref{Prop1}, we get $H_{\mathscr{U}}(\hat{x})\leq_{K}^{L} H_\mathscr{U}(x^*).$ This is a contradiction to the fact that $x^*$ is a $\leq_{K}^{L}$-robust solution. 

    Conversely, assume that $x^*\in \mathcal{S}$ is not a $\leq_{K}^{L}$-robust solution for (\ref{SP(u)}). Then there exists $x\in \mathcal{S}\setminus \{x^*\}$ such that $H_{\mathscr{U}}(x)\leq_{K}^{L} H_{\mathscr{U}}(x^*).$ This implies that for any $a\in \mathscr{U},$ there exists $u\in\mathscr{U}$ such that $H(x, u) \leq_{K}^{l} H(x^*, a).$ By Theorem \ref{new lu relation via scalarization} (i), $\mathbb{Z}_1^{e, K}(H(x, u), H(x^*, a))\leq 0$ for all $e\in K\setminus \{0\}$. Consequently, $$\underset{a\in \mathscr{U}}{\sup} \underset{u\in \mathscr{U}}{\inf} \mathbb{Z}_1^{e, K}(H(x, u), H(x^*, a))\leq 0.$$ This implies that $\hat{R}^2_{x^*,e}\cap C\neq \emptyset.$

\end{proof}

For $\leq_K^u$-robust solution, some similar result has been derived in Proposition 4.2, and Theorem 4.2 of \cite{Das2024}. However, they also need to be modified. Recall Proposition 4.2 from \cite{Das2024}.
\begin{proposition}\label{Prop4.2 of Das}
    Consider problem (\ref{SP(u)}). Let $x^*\in \mathcal{S}.$ Assume that for any $u\in \mathscr{U}$ and $x\in \mathscr{X}\setminus \{x^*\},$ there exists $q_u\in H_{\mathscr{U}}(x^*)$ such that $H(x, u)\subseteq q_u-K.$ If $x^*$ is a $\leq_{K}^{u}$ robust solution, then $R^3_{x^*}\cap C=\emptyset,$ or equivalently, the generalized system $B^3_{x^*}(x)\in C,$ $x\in \mathcal{S}\setminus\{x^*\}$ is inconsistent. 
\end{proposition}
\begin{remark}
 Here again, the assumption that for any $u\in \mathscr{U}$ and $x\in \mathscr{X}\setminus \{x^*\},$ there exists $q_u\in H_{\mathscr{U}}(x^*)$ such that $H(x, u)\subseteq q_u-K,$ implies $H_{\mathscr{U}}(x)\leq_{K}^{u} H_{\mathscr{U}}(x^*)$ for all $x\in \mathscr{X}\setminus \{x^*\},$ and hence $x^*$ cannot be a $\leq_{K}^{u}$ robust solution. We propose the following modification using $\hat{R}^3_{x^*,e}$. Please note that we consider only those situations where $\underset{u\in \mathscr{U}}{\sup } \mathscr{Z}_2^{e, K} (H(x,u))$ is finite valued to make $\hat{B}^3_{x^*,e}(x)$ properly defined.    
\end{remark}

\begin{proposition}\label{Prop4.2 of Das renewed}
    Consider problem (\ref{SP(u)}). Let $x^*\in \mathcal{S}$ and $e\in K \setminus \{0\}$. If $\hat{R}^3_{x^*,e}\cap C=\emptyset$ (equivalently $\hat{B}^3_{x^*,e}(x)\in C,$ $x\in \mathcal{S}\setminus\{x^*\}$ is inconsistent), then $x^*$ is a $\leq_{K}^{U}$ robust solution.
\end{proposition}
\begin{proof}
    Assume, if possible, that $x^*\in \mathcal{S}$ is not a $\leq_{K}^{U}$ robust solution. Then there exists $x\in \mathcal{S}\setminus \{x^*\}$ such that $H_{\mathscr{U}}(x)\leq_{K}^{U} H_{\mathscr{U}}(x^*),$ i.e., $\underset{u\in \mathscr{U}}{\bigcup} H(x,u)\leq_{K}^{U} \underset{a\in \mathscr{U}}{\bigcup} H(x^*,a).$ By Proposition \ref{Mod:Prop3.7}, we get $\underset{u\in \mathscr{U}}{\sup}~\mathscr{Z}_2^{e, K}(H(x,u))\leq \underset{a\in \mathscr{U}}{\sup}~\mathscr{Z}_2^{e, K}(H(x^*,a)).$ It follows that $x\in \hat{R}^3_{x^*,e}\cap C$. This contradicts the assumption. 
   
\end{proof}
Now we recall Theorem 4.2 of \cite{Das2024}.
\begin{theorem}\label{Thm4.2 of Das}
    Let $x^*\in \mathcal{S}.$ Assume that 
    \begin{itemize}
        \item[\normalfont{(i)}] for any $u\in \mathscr{U}$ and $x\in \mathscr{X}\setminus \{x^*\},$ there exists an element $q_u\in H_{\mathscr{U}}(x^*)$ such that $H(x, u)\subseteq q_u-K;$
        \item[\normalfont{(ii)}] there exists $e\in K\setminus \{0\}$ such that the infimum $\underset{a\in \mathscr{U}}{\inf} \mathbb{G}^{e, K} (H(x^*, a), H(x,u))$ is attained for all $u\in \mathscr{U};$ and 
        \item[\normalfont{(iii)}] $H(x,u)$ is $-K$-closed for all $x$ and $u.$
    \end{itemize}
    Then $x^*$ is a $\leq_{K}^{u}$-robust solution for (\ref{SP(u)}) if and only if $R^4_{x^*}\cap C=\emptyset,$ or equivalently, the generalized system $B^4_{x^*}(x)\in C$ is inconsistent. 
\end{theorem}
\begin{remark}
    As usual, the assumption (i) would imply $x^*$ cannot be a $\leq_{K}^{u}$ robust solution. So we propose the following modification.
\end{remark}

\begin{theorem}\label{Main2.0}
    Consider problem (\ref{SP(u)}). Let $x^*\in \mathcal{S}.$ Assume that
    \begin{itemize}
        \item[\normalfont{(i)}]  $H(x, u)-K$ is closed for every $x$ in $\mathscr{X}$ and $u$ in $\mathscr{U}$;
        \item[\normalfont{(ii)}]  there exists $e\in K\setminus \{0\}$ such that $\underset{a\in \mathscr{U}}{\inf} \mathbb{Z}_2^{e,K}(H(x,u), H(x^*, a))$ is attained for all $x \in \mathscr{X}\setminus\{x^*\}$ and $u \in \mathscr{U}.$
    \end{itemize}
    Then $x^*$ is a $\leq_{K}^{U}$-robust solution for (\ref{SP(u)}) if and only if $\hat{R}^4_{x^*,e}\cap C=\emptyset.$
\end{theorem}
\begin{proof}
    Let $x^*\in \mathcal{S}$ be a $\leq_{K}^{U}$-robust solution. 
    On the contrary, assume that $\hat{R}^4_{x^*,e}\cap C\neq \emptyset$. Let $\bar{v} \in \hat{R}^4_{x^*,e}\cap C.$ Then $ \bar{v} =\hat{B}_{x^*,e}^4(\bar{x})$ for some $\bar{x}\in \mathscr{X}\setminus \{x^*\}.$ Clearly, $\bar{x}\in \mathcal{S}\setminus \{x^*\}$ and $\underset{u\in \mathscr{U}}{\sup}\underset{a\in \mathscr{U}}{\inf}~ \mathbb{Z}_2^{e,K}( H(\bar{x},u), H(x^*, a))\leq 0.$ By Proposition \ref{Prop2}, we get $H_{\mathscr{U}}(\bar{x})\leq_{K}^{U} H_\mathscr{U}(x^*).$ This is a contradiction. 

    Conversely, assume that $x^*$ is not a $\leq_{K}^{U}$-robust solution for (\ref{SP(u)}). Then there exists $x\in \mathcal{S}\setminus \{x^*\}$ such that $H_{\mathscr{U}}(x)\leq_{K}^{U} H_{\mathscr{U}}(x^*).$ This implies that for any $u\in \mathscr{U},$ there exists $a\in\mathscr{U}$ such that $H(x, u) \leq_{K}^{u} H(x^*, a).$ By Theorem \ref{new lu relation via scalarization} (ii), $\mathbb{Z}_2^{e, K}(H(x, u), H(x^*, a))\leq 0.$ Consequently, $$\displaystyle \sup_{u\in \mathscr{U}} \inf_{a\in \mathscr{U}} \mathbb{Z}_2^{e, K}(H(x, u), H(x^*, a))\leq 0.$$ This implies that $\hat{R}^4_{x^*,e}\cap C\neq \emptyset.$
    \end{proof}

    For $\leq_K^s$ robust solution, we recall Theorem 4.3 of \cite{Das2024}.
\begin{theorem}\label{Thm4.3 of Das}
    Consider problem (\ref{SP(u)}). Let $x^*\in \mathcal{S}.$ Assume that 
    \begin{itemize}
        \item[\normalfont{(i)}] for any $u\in \mathscr{U}$ and $x\in \mathscr{X}\setminus \{x^*\},$ there exist $p_u\in  H_{\mathscr{U}}(x)$ such that $H(x^*, u)\subseteq p_u+K$ and $q_u\in H_{\mathscr{U}}(x^*)$ such that $H(x, u)\subseteq q_u-K;$
        \item[\normalfont{(ii)}] there exists $e\in K\setminus \{0\}$ such that the infimum $\underset{u\in \mathscr{U}}{\inf} \mathbb{G}^{e, -K} (H(x,u), H(x^*, a))$ is attained for all $a\in \mathscr{U}$ and the infimum $\underset{a\in \mathscr{U}}{\inf} \mathbb{G}^{e, K} (H(x^*, a), H(x,u))$ is attained for all $u\in \mathscr{U};$ and 
        \item[\normalfont{(iii)}] $H(x,u)$ is both $K$-closed and $-K$-closed for all $x$ and $u.$
    \end{itemize}
    Then $x^*$ is a $\leq_{K}^{s}$-robust solution for (\ref{SP(u)}) if and only if $R^5_{x^*}\cap C=\emptyset,$ or equivalently, the generalized system $B^5_{x^*}(x)\in C$ is inconsistent. 
\end{theorem}
\begin{remark}
    From assumption (i) of the above theorem, it follows that $H_{\mathscr{U}}(x)\leq_{K}^{l} H_{\mathscr{U}}(x^*)$ and $H_{\mathscr{U}}(x)\leq_{K}^{u} H_{\mathscr{U}}(x^*)$ for any $x\in \mathscr{X}\setminus \{x^*\}.$ Consequently, $x^*$ cannot be a $\leq_{K}^{s}$ robust solution of (\ref{SP(u)}).  Also, as discussed earlier, the expression $\underset{u\in \mathscr{U}}{\inf} \mathbb{G}^{e, -K} (H(x,u), H(x^*, a))$ in assumption (ii) suggests that $e$ must be in $-K\setminus \{0\}.$ Therefore, we propose the following modification of Theorem \ref{Thm4.3 of Das}.
\end{remark}

\begin{theorem}\label{Thm4.3 of Das renewed}
    Consider problem (\ref{SP(u)}). Let $x^*\in \mathcal{S}.$ Assume that
    \begin{itemize}
        \item[\normalfont{(i)}] $H(x, u)+K$ and $H(x, u)-K$ both are closed for every $x$ in $\mathscr{X}$ and $u$ in $\mathscr{U}$;
        \item[\normalfont{(ii)}]  there exists $e\in K\setminus \{0\}$ such that $\underset{u\in \mathscr{U}}{\inf}~ \mathbb{Z}_1^{e,K}(H(x,u), H(x^*, a))$ is attained for all $x \in \mathscr{X}\setminus\{x^*\}$ and for all $a \in \mathscr{U};$ and 
        $\underset{a\in \mathscr{U}}{\inf} \mathbb{Z}_2^{e,K}(H(x,u), H(x^*, a))$ is attained for all $x \in \mathscr{X}\setminus\{x^*\}$ and for all $u \in \mathscr{U}.$
    \end{itemize}
    Then $x^*$ is a $\leq_{K}^{S}$-robust solution for (\ref{SP(u)}) if and only if $\hat{R}^5_{x^*,e}\cap C=\emptyset.$
\end{theorem}
\begin{proof}
    Assume that $x^*$ is a $\leq_{K}^{S}$-robust solution for (\ref{SP(u)}) but $\hat{R}^5_{x^*,e}\cap C\neq \emptyset$. Let $\bar{v}\in \hat{R}^5_{x^*, e}\cap C.$ Then $\bar{v}= \hat{B}^5_{x^*,e}(\bar{x})$ for some $\bar{x}\in \mathscr{X}\setminus \{x^*\}.$ It follows that $\bar{x}\in \mathcal{S}\setminus \{x^*\}$, $\underset{a\in \mathscr{U}}{\sup } \underset{u\in \mathscr{U}}{\inf } \mathbb{Z}_1^{e, K} (H(\bar{x},u), H(x^*, a))\leq 0$ and $\underset{u\in \mathscr{U}}{\sup } \underset{a\in \mathscr{U}}{\inf } \mathbb{Z}_2^{e, K} (H(\bar{x},u), H(x^*, a))\leq 0.$ Thus by Proposition \ref{Prop1} and Proposition \ref{Prop2} we get $H_{\mathscr{U}}(\bar{x})\leq_{K}^{L} H_{\mathscr{U}}(x^*)$ and $H_{\mathscr{U}}(\bar{x})\leq_{K}^{U} H_{\mathscr{U}}(x^*).$ In other words, for every $a\in \mathscr{U},$ there exists $u\in \mathscr{U}$ such that $H(\bar{x}, u)\leq_{K}^{l} H(x^*,a)$ and for every $u'\in \mathscr{U},$ there exists $a'\in \mathscr{U}$ such that $H(\bar{x}, u')\leq_{K}^{u} H(x^*,a').$ This shows that $H_{\mathscr{U}}(x)\leq_{K}^{S} H_{\mathscr{U}}(x^*).$ Thus $x^*$ is not a $\leq_{K}^{S}$-robust solution for (\ref{SP(u)}), a contradiction.

    Conversely, assume that $x^*$ is not a $\leq_{K}^{S}$-robust solution of (\ref{SP(u)}). Then there exists $\bar{y}\in \mathcal{S}\setminus \{x^*\}$ such that $H_{\mathscr{U}}(\bar{y})\leq_{K}^{S} H_{\mathscr{U}}(x^*).$ That is, $H_{\mathscr{U}}(\bar{y})\leq_{K}^{L} H_{\mathscr{U}}(x^*)$ and $H_{\mathscr{U}}(\bar{y})\leq_{K}^{U} H_{\mathscr{U}}(x^*).$ By Theorem \ref{Main2} and \ref{Main2.0}, we get $\hat{R}^2_{x^*,e}\cap C\neq \emptyset$ and $\hat{R}^4_{x^*,e}\cap C\neq \emptyset.$ Consequently $\hat{R}^5_{x^*,e}\cap C\neq \emptyset.$ This completes the proof.
\end{proof}

We now propose some more necessary conditions for various robust solutions via scalarization where a dual cone is used. These conditions are applicable when the involved set-valued objective map is convex-valued. Let $x^*\in \mathscr{X}$. For $w\in$ $K^*\setminus \{0\},$ let $e_w\in \operatorname{int}(K_w).$ Define $\tilde{B}_{x^*, e_w}^1: \mathscr{X} \to \mathbb{R}^{1+m}$ by $$\tilde{B}_{x^*, e_w}^1(x)=\left(\inf_{u\in \mathscr{U}}\mathscr{Z}_1^{e_w, K_w}(H(x, u))- \inf_{a\in \mathscr{U}}\mathscr{Z}_1^{e_w, K_w} (H(x^*, a)), \mathscr{F}(x)\right).$$ We consider the sets $$\tilde{R}_{x^*, e_w}^1=\left\{\tilde{B}_{x^*,e_w}^1(x): x\in \mathscr{X}\setminus \{x^*\}\right\}$$ and $$C'=\left\{(u, v)\in \mathbb{R}^{1+m}: u< 0, v<0 \right\}.$$ Of course we will only consider the situations when $\underset{u\in \mathscr{U}}{\inf}\mathscr{Z}_1^{e_w, K_w}(H(x, u))$ is finite for all $x\in\mathscr{X}$ so that the definitions of $\tilde{B}_{x^*, e_w}^1$ and $\tilde{R}_{x^*, e_w}^1$ are meaningful.
\begin{theorem}\label{Main1}
     Consider problem (\ref{SP(u)}). Assume that $H(x, u)+K$ is closed and convex for all $x\in \mathscr{X}$ and for all $u\in \mathscr{U}$. If $x^*$ is a $\leq_{K}^{L}$-robust solution for (\ref{SP(u)}), then there exists $w\in K^*\setminus \{0\}$ such that $\tilde{R}_{x^*, e_w}^1\cap C'=\emptyset$ for all $e_w\in \operatorname{int}(K_w).$
\end{theorem}
   
\begin{proof}
    Let $x^*\in \mathcal{S}$ be a $\leq_{K}^{L}$-robust solution. 
    On the contrary assume that for all $w\in K^*\setminus \{0\}$, there exists $e_w\in$ int$(K_w)$ such that $\tilde{R}_{x^*, e_w}^1\cap C'\neq \emptyset.$  Fix $w\in K^*\setminus \{0\}$. Then $\tilde{R}_{x^*, e_w}^1\cap C'\neq \emptyset$ for some $e_w\in$ int$(K_w)$. Let $\bar{v} \in \tilde{R}_{x^*, e_w}^1\cap C'.$ Then $\bar{v}=\tilde{B}_{x^*,e_w}^1(\bar{x})$ for some $\bar{x} \in \mathscr{X}\setminus \{x^*\}.$ Consequently, $\underset{u\in \mathscr{U}}{\inf}\mathscr{Z}_1^{e_w, K_w}(H(\bar{x}, u))< \underset{a\in \mathscr{U}}{\inf}~\mathscr{Z}_1^{e_w, K_w} (H(x^*, a))$ and hence $\underset{u\in \mathscr{U}}{\inf}\mathscr{Z}_1^{e_w, K_w}(H(\bar{x}, u))< \mathscr{Z}_1^{e_w, K_w} (H(x^*, a))$ for all $a\in \mathscr{U}$. Fix $\bar{a}\in \mathscr{U}$. Then $ \underset{u\in \mathscr{U}}{\inf}~\mathscr{Z}_1^{e_w, K_w} (H(x, u))< \mathscr{Z}_1^{e_w, K_w} (H(x^*, \bar{a}))$. By the definition of infimum, there exists $\bar{u}\in \mathscr{U}$ such that $\mathscr{Z}_1^{e_w, K_w}(H(\bar{x}, \bar{u}))<  \mathscr{Z}_1^{e_w, K_w} (H(x^*, \bar{a}))$. Therefore, for every $\bar{a}\in \mathscr{U}$, there exists $\bar{u}\in \mathscr{U}$ such that $\mathscr{Z}_1^{e_w, K_w}(H(\bar{x}, \bar{u}))< \mathscr{Z}_1^{e_w, K_w} (H(x^*, \bar{a}))$. By Theorem \ref{Kobis Thm 3.9}, we get $H_\mathscr{U}(\bar{x})\leq_{K}^{L} H_\mathscr{U}(x^*).$ This is a contradiction.
\end{proof}

For a $\leq_K^U$-robust solution, a similar result can be derived. Let $x^*\in \mathscr{X}$. Let $w\in$ $K^*\setminus \{0\},$ and $e_w\in \operatorname{int}(K_w).$ Define $\tilde{B}_{x^*, e_w}^2: \mathscr{X} \to \mathbb{R}^{1+m}$ by $$\tilde{B}_{x^*, e_w}^2(x)=\left(\sup_{u\in \mathscr{U}}\mathscr{Z}_2^{e_w, K_w}(H(x, u))- \sup_{a\in \mathscr{U}}\mathscr{Z}_2^{e_w, K_w} (H(x^*, a)), \mathscr{F}(x)\right).$$ We consider the set $\tilde{R}_{x^*, e_w}^2=\left\{\tilde{B}_{x^*,e_w}^2(x): x\in \mathscr{X}\setminus \{x^*\}\right\}.$ 

\begin{theorem}\label{Main1}
     Let $x^*\in \mathcal{S}$ be such that $H(x^*, a)-K$ is closed and convex for every $a\in \mathscr{U}$. If $x^*$ is a $\leq_{K}^{U}$-robust solution for (\ref{SP(u)}), then there exists $w\in K^*\setminus \{0\}$ such that $\tilde{R}_{x^*, e_w}^2\cap C'=\emptyset$ for all $e_w\in \operatorname{int}(K_w).$
\end{theorem}
   
\begin{proof}
    Let $x^*\in \mathcal{S}$ be a $\leq_{K}^{U}$-robust solution.  
    On the contrary assume that for all $w\in K^*\setminus \{0\}$, there exists $e_w\in$ int$(K_w)$ such that $\tilde{R}_{x^*, e_w}^2\cap C'\neq \emptyset.$  Fix $w\in K^*\setminus \{0\}$. Then $\tilde{R}_{x^*, e_w}^2\cap C'\neq \emptyset$ for some $e_w\in$ int$(K_w)$. Let $\bar{v} \in \tilde{R}_{x^*, e_w}^2\cap C'.$ Then $\bar{v}=\tilde{B}_{x^*,e_w}^2(\bar{x})$ for some $\bar{x} \in \mathscr{X}\setminus \{x^*\}.$ Consequently, $\underset{u\in \mathscr{U}}{\sup}\mathscr{Z}_2^{e_w, K_w}(H(\bar{x}, u))< \underset{a\in \mathscr{U}}{\sup}\mathscr{Z}_2^{e_w, K_w} (H(x^*, a))$. This implies that  $\mathscr{Z}_2^{e_w, K_w}(H(\bar{x}, u))< \underset{u\in \mathscr{U}}{\sup}\mathscr{Z}_2^{e_w, K_w} (H(x^*, a))$ for all $u\in \mathscr{U}.$ Therefore, for every $\bar{u}\in \mathscr{U}$, there exists $\bar{a}\in \mathscr{U}$ such that $\mathscr{Z}_2^{e_w, K_w}(H(\bar{x}, \bar{u}))< \mathscr{Z}_2^{e_w, K_w} (H(x^*, \bar{a}))$. By Theorem \ref{Kobis Thm 3.6}, we get $H_\mathscr{U}(\bar{x})\leq_{K}^{U} H_\mathscr{U}(x^*).$ This is a contradiction.
\end{proof}

\section{Conclusion}
In this paper, we have critically analyzed the existing results in the literature on the robust solution to uncertain set-valued optimization problems, particularly those found in the work of \cite{Das2024}. By pointing out several inconsistencies in their findings, we have proposed necessary modifications and introduced novel concepts in the formulation of robust solutions for such problems. In the process, we have also identified some issues for scalarization functions for general set-valued literature and improved upon some of the results. 

\backmatter

\bmhead{Acknowledgments} 

\section*{Declarations}
\begin{itemize}
\item Conflict of interest/Competing interests : There is no conflict of interest among the authors.
\item Availability of data and materials : NA
\item Code availability : NA
\item Authors' contributions : All the authors have contributed equally.
\end{itemize}


\bibliography{sn-bibliography}


\begin{thebibliography}{36}
\ifx \bisbn   \undefined \def \bisbn  #1{ISBN #1}\fi
\ifx \binits  \undefined \def \binits#1{#1}\fi
\ifx \bauthor  \undefined \def \bauthor#1{#1}\fi
\ifx \batitle  \undefined \def \batitle#1{#1}\fi
\ifx \bjtitle  \undefined \def \bjtitle#1{#1}\fi
\ifx \bvolume  \undefined \def \bvolume#1{\textbf{#1}}\fi
\ifx \byear  \undefined \def \byear#1{#1}\fi
\ifx \bissue  \undefined \def \bissue#1{#1}\fi
\ifx \bfpage  \undefined \def \bfpage#1{#1}\fi
\ifx \blpage  \undefined \def \blpage #1{#1}\fi
\ifx \burl  \undefined \def \burl#1{\textsf{#1}}\fi
\ifx \doiurl  \undefined \def \doiurl#1{\url{https://doi.org/#1}}\fi
\ifx \betal  \undefined \def \betal{\textit{et al.}}\fi
\ifx \binstitute  \undefined \def \binstitute#1{#1}\fi
\ifx \binstitutionaled  \undefined \def \binstitutionaled#1{#1}\fi
\ifx \bctitle  \undefined \def \bctitle#1{#1}\fi
\ifx \beditor  \undefined \def \beditor#1{#1}\fi
\ifx \bpublisher  \undefined \def \bpublisher#1{#1}\fi
\ifx \bbtitle  \undefined \def \bbtitle#1{#1}\fi
\ifx \bedition  \undefined \def \bedition#1{#1}\fi
\ifx \bseriesno  \undefined \def \bseriesno#1{#1}\fi
\ifx \blocation  \undefined \def \blocation#1{#1}\fi
\ifx \bsertitle  \undefined \def \bsertitle#1{#1}\fi
\ifx \bsnm \undefined \def \bsnm#1{#1}\fi
\ifx \bsuffix \undefined \def \bsuffix#1{#1}\fi
\ifx \bparticle \undefined \def \bparticle#1{#1}\fi
\ifx \barticle \undefined \def \barticle#1{#1}\fi
\bibcommenthead
\ifx \bconfdate \undefined \def \bconfdate #1{#1}\fi
\ifx \botherref \undefined \def \botherref #1{#1}\fi
\ifx \url \undefined \def \url#1{\textsf{#1}}\fi
\ifx \bchapter \undefined \def \bchapter#1{#1}\fi
\ifx \bbook \undefined \def \bbook#1{#1}\fi
\ifx \bcomment \undefined \def \bcomment#1{#1}\fi
\ifx \oauthor \undefined \def \oauthor#1{#1}\fi
\ifx \citeauthoryear \undefined \def \citeauthoryear#1{#1}\fi
\ifx \endbibitem  \undefined \def \endbibitem {}\fi
\ifx \bconflocation  \undefined \def \bconflocation#1{#1}\fi
\ifx \arxivurl  \undefined \def \arxivurl#1{\textsf{#1}}\fi
\csname PreBibitemsHook\endcsname

\bibitem{Das2024}
\begin{barticle}
\bauthor{\bsnm{Das}, \binits{M.}},
\bauthor{\bsnm{Nahak}, \binits{C.}},
\bauthor{\bsnm{Biswal}, \binits{M.P.}}:
\batitle{Treatment of set-valued robustness via separation and scalarization}.
\bjtitle{J. Optim. Theory Appl.}
\bvolume{201}(\bissue{2}),
\bfpage{843}--\blpage{865}
(\byear{2024}).
\doiurl{10.1007/s10957-024-02423-4}
\end{barticle}
\endbibitem

\bibitem{borwein1977multivalued}
\begin{barticle}
\bauthor{\bsnm{Borwein}, \binits{J.}}:
\batitle{Multivalued convexity and optimization: a unified approach to inequality and equality constraints}.
\bjtitle{Math. Programming}
\bvolume{13}(\bissue{2}),
\bfpage{183}--\blpage{199}
(\byear{1977}).
\doiurl{10.1007/BF01584336}
\end{barticle}
\endbibitem

\bibitem{corley1988optimality}
\begin{barticle}
\bauthor{\bsnm{Corley}, \binits{H.W.}}:
\batitle{Optimality conditions for maximizations of set-valued functions}.
\bjtitle{J. Optim. Theory Appl.}
\bvolume{58}(\bissue{1}),
\bfpage{1}--\blpage{10}
(\byear{1988}).
\doiurl{10.1007/BF00939767}
\end{barticle}
\endbibitem

\bibitem{DK4}
\begin{barticle}
\bauthor{\bsnm{Kuroiwa}, \binits{D.}}:
\batitle{Existence theorems of set optimization with set-valued maps}.
\bjtitle{J. Inf. Optim. Sci.}
\bvolume{24}(\bissue{1}),
\bfpage{73}--\blpage{84}
(\byear{2003}).
\doiurl{10.1080/02522667.2003.10699556}
\end{barticle}
\endbibitem

\bibitem{Akhtar}
\begin{bbook}
\bauthor{\bsnm{Khan}, \binits{A.A.}},
\bauthor{\bsnm{Tammer}, \binits{C.}},
\bauthor{\bsnm{Z\u{a}linescu}, \binits{C.}}:
\bbtitle{Set-valued optimization---An Introduction with Applications}.
\bsertitle{Vector Optimization},
p. \bfpage{765}.
\bpublisher{Springer},
\blocation{Heidelberg}
(\byear{2015}).
\doiurl{10.1007/978-3-642-54265-7}
\end{bbook}
\endbibitem

\bibitem{Gofert}
\begin{bbook}
\bauthor{\bsnm{G\"{o}pfert}, \binits{A.}},
\bauthor{\bsnm{Riahi}, \binits{H.}},
\bauthor{\bsnm{Tammer}, \binits{C.}},
\bauthor{\bsnm{Z\u{a}linescu}, \binits{C.}}:
\bbtitle{Variational Methods in Partially Ordered Spaces}.
\bsertitle{CMS Books in Mathematics/Ouvrages de Math\'{e}matiques de la SMC},
vol. \bseriesno{17},
p. \bfpage{350}.
\bpublisher{Springer},
\blocation{New York}
(\byear{2003}).
\doiurl{10.1007/b97568}
\end{bbook}
\endbibitem

\bibitem{corley1987existence}
\begin{barticle}
\bauthor{\bsnm{Corley}, \binits{H.W.}}:
\batitle{Existence and {L}agrangian duality for maximizations of set-valued functions}.
\bjtitle{J. Optim. Theory Appl.}
\bvolume{54}(\bissue{3}),
\bfpage{489}--\blpage{501}
(\byear{1987}).
\doiurl{10.1007/BF00940198}
\end{barticle}
\endbibitem

\bibitem{Luc89}
\begin{bbook}
\bauthor{\bsnm{Luc}, \binits{D.T.}}:
\bbtitle{Theory of Vector Optimization}.
\bsertitle{Lecture Notes in Economics and Mathematical Systems},
vol. \bseriesno{319},
p. \bfpage{173}.
\bpublisher{Springer},
\blocation{Berlin}
(\byear{1989}).
\doiurl{10.1007/978-3-642-50280-4}
\end{bbook}
\endbibitem

\bibitem{DK}
\begin{bchapter}
\bauthor{\bsnm{Kuroiwa}, \binits{D.}}:
\bctitle{On set-valued optimization}.
In: \bbtitle{Proceedings of the {T}hird {W}orld {C}ongress of {N}onlinear {A}nalysts, {P}art 2 ({C}atania, 2000)},
vol. \bseriesno{47},
pp. \bfpage{1395}--\blpage{1400}
(\byear{2001}).
\doiurl{10.1016/S0362-546X(01)00274-7}
\end{bchapter}
\endbibitem

\bibitem{KurTanTru97}
\begin{bchapter}
\bauthor{\bsnm{Kuroiwa}, \binits{D.}},
\bauthor{\bsnm{Tanaka}, \binits{T.}},
\bauthor{\bsnm{Ha}, \binits{T.X.D.}}:
\bctitle{On cone convexity of set-valued maps}.
In: \bbtitle{Proceedings of the {S}econd {W}orld {C}ongress of {N}onlinear {A}nalysts, {P}art 3 ({A}thens, 1996)},
vol. \bseriesno{30},
pp. \bfpage{1487}--\blpage{1496}
(\byear{1997}).
\doiurl{10.1016/S0362-546X(97)00213-7}
\end{bchapter}
\endbibitem

\bibitem{alonso2005set}
\begin{barticle}
\bauthor{\bsnm{Alonso}, \binits{M.}},
\bauthor{\bsnm{Rodr\'{\i}guez-Mar\'{\i}n}, \binits{L.}}:
\batitle{Set-relations and optimality conditions in set-valued maps}.
\bjtitle{Nonlinear Anal.}
\bvolume{63}(\bissue{8}),
\bfpage{1167}--\blpage{1179}
(\byear{2005}).
\doiurl{10.1016/j.na.2005.06.002}
\end{barticle}
\endbibitem

\bibitem{hernandez2007existence}
\begin{barticle}
\bauthor{\bsnm{Hern\'{a}ndez}, \binits{E.}},
\bauthor{\bsnm{Rodr\'{\i}guez-Mar\'{\i}n}, \binits{L.}}:
\batitle{Existence theorems for set optimization problems}.
\bjtitle{Nonlinear Anal.}
\bvolume{67}(\bissue{6}),
\bfpage{1726}--\blpage{1736}
(\byear{2007}).
\doiurl{10.1016/j.na.2006.08.013}
\end{barticle}
\endbibitem

\bibitem{kuroiwa2003weightedcriteriaforexistence}
\begin{barticle}
\bauthor{\bsnm{Kuroiwa}, \binits{D.}}:
\batitle{Existence of efficient points of set optimization with weighted criteria}.
\bjtitle{J. Nonlinear Convex Anal.}
\bvolume{4}(\bissue{1}),
\bfpage{117}--\blpage{123}
(\byear{2003})
\end{barticle}
\endbibitem

\bibitem{vetrivelkuntalbookchater}
\begin{bchapter}
\bauthor{\bsnm{Som}, \binits{K.}},
\bauthor{\bsnm{Vetrivel}, \binits{V.}}:
\bctitle{Results on existence of $l$-minimal and $u$-minimal solutions in set-valued optimization: A brief survey}.
In: \bbtitle{Convex Optimization- Theory, Algorithms, and Applications}.
\bsertitle{Springer Proc.}
\bpublisher{Springer},
\blocation{Singapore}
(\byear{2025}).
\doiurl{10.1007/978-981-97-8907-8_5}
\end{bchapter}
\endbibitem

\bibitem{minimalelementHamel2006}
\begin{barticle}
\bauthor{\bsnm{Hamel}, \binits{A.}},
\bauthor{\bsnm{L\"{o}hne}, \binits{A.}}:
\batitle{Minimal element theorems and {E}keland's principle with set relations}.
\bjtitle{J. Nonlinear Convex Anal.}
\bvolume{7}(\bissue{1}),
\bfpage{19}--\blpage{37}
(\byear{2006})
\end{barticle}
\endbibitem

\bibitem{hernandez2007nonconvex}
\begin{barticle}
\bauthor{\bsnm{Hern{\'a}ndez}, \binits{E.}},
\bauthor{\bsnm{Rodr{\'\i}guez-Mar{\'\i}n}, \binits{L.}}:
\batitle{Nonconvex scalarization in set optimization with set-valued maps}.
\bjtitle{Journal of Mathematical analysis and Applications}
\bvolume{325}(\bissue{1}),
\bfpage{1}--\blpage{18}
(\byear{2007}).
\doiurl{10.1016/j.jmaa.2006.01.033}
\end{barticle}
\endbibitem

\bibitem{guttierez2015scalarizationsurvey}
\begin{bchapter}
\bauthor{\bsnm{Guti{\'e}rrez}, \binits{C.}},
\bauthor{\bsnm{Jim{\'e}nez}, \binits{B.}},
\bauthor{\bsnm{Novo}, \binits{V.}}:
\bctitle{Nonlinear scalarizations of set optimization problems with set orderings}.
In: \beditor{\bsnm{Hamel}, \binits{A.H.}},
\beditor{\bsnm{Heyde}, \binits{F.}},
\beditor{\bsnm{L{\"o}hne}, \binits{A.}},
\beditor{\bsnm{Rudloff}, \binits{B.}},
\beditor{\bsnm{Schrage}, \binits{C.}} (eds.)
\bbtitle{Set Optimization and Applications - The State of the Art},
pp. \bfpage{43}--\blpage{63}.
\bpublisher{Springer},
\blocation{Berlin, Heidelberg}
(\byear{2015}).
\doiurl{10.1007/978-3-662-48670-2_2}
\end{bchapter}
\endbibitem

\bibitem{Pradeeprecentsurvey}
\begin{bchapter}
\bauthor{\bsnm{Ansari}, \binits{Q.H.}},
\bauthor{\bsnm{Sharma}, \binits{P.K.}}:
\bctitle{Set order relations, set optimization, and {E}keland's variational principle}.
In: \bbtitle{Optimization, Variational Analysis and Applications}.
\bsertitle{Springer Proc. Math. Stat.},
vol. \bseriesno{355},
pp. \bfpage{103}--\blpage{165}.
\bpublisher{Springer},
\blocation{Singapore}
(\byear{2021}).
\doiurl{10.1007/978-981-16-1819-2\_6}
\end{bchapter}
\endbibitem

\bibitem{gerth1990nonconvex}
\begin{barticle}
\bauthor{\bsnm{Gerth}, \binits{C.}},
\bauthor{\bsnm{Weidner}, \binits{P.}}:
\batitle{Nonconvex separation theorems and some applications in vector optimization}.
\bjtitle{J. Optim. Theory Appl.}
\bvolume{67}(\bissue{2}),
\bfpage{297}--\blpage{320}
(\byear{1990}).
\doiurl{10.1007/BF00940478}
\end{barticle}
\endbibitem

\bibitem{Kobis2016}
\begin{barticle}
\bauthor{\bsnm{K\"obis}, \binits{E.}},
\bauthor{\bsnm{K\"obis}, \binits{M.A.}}:
\batitle{Treatment of set order relations by means of a nonlinear scalarization functional: a full characterization}.
\bjtitle{Optimization}
\bvolume{65}(\bissue{10}),
\bfpage{1805}--\blpage{1827}
(\byear{2016}).
\doiurl{10.1080/02331934.2016.1219355}
\end{barticle}
\endbibitem

\bibitem{Zhang2009}
\begin{barticle}
\bauthor{\bsnm{Zhang}, \binits{W.Y.}},
\bauthor{\bsnm{Li}, \binits{S.J.}},
\bauthor{\bsnm{Teo}, \binits{K.L.}}:
\batitle{Well-posedness for set optimization problems}.
\bjtitle{Nonlinear Anal.}
\bvolume{71}(\bissue{9}),
\bfpage{3769}--\blpage{3778}
(\byear{2009}).
\doiurl{10.1016/j.na.2009.02.036}
\end{barticle}
\endbibitem

\bibitem{Gutierrez2012}
\begin{barticle}
\bauthor{\bsnm{Guti\'{e}rrez}, \binits{C.}},
\bauthor{\bsnm{Miglierina}, \binits{E.}},
\bauthor{\bsnm{Molho}, \binits{E.}},
\bauthor{\bsnm{Novo}, \binits{V.}}:
\batitle{Pointwise well-posedness in set optimization with cone proper sets}.
\bjtitle{Nonlinear Anal.}
\bvolume{75}(\bissue{4}),
\bfpage{1822}--\blpage{1833}
(\byear{2012}).
\doiurl{10.1016/j.na.2011.09.028}
\end{barticle}
\endbibitem

\bibitem{Han2017}
\begin{barticle}
\bauthor{\bsnm{Han}, \binits{Y.}},
\bauthor{\bsnm{Huang}, \binits{N.-j.}}:
\batitle{Well-posedness and stability of solutions for set optimization problems}.
\bjtitle{Optimization}
\bvolume{66}(\bissue{1}),
\bfpage{17}--\blpage{33}
(\byear{2017}).
\doiurl{10.1080/02331934.2016.1247270}
\end{barticle}
\endbibitem

\bibitem{Pap2}
\begin{barticle}
\bauthor{\bsnm{Som}, \binits{K.}},
\bauthor{\bsnm{Vetrivel}, \binits{V.}}:
\batitle{A note on pointwise well-posedness of set-valued optimization problems}.
\bjtitle{J. Optim. Theory Appl.}
\bvolume{192}(\bissue{2}),
\bfpage{628}--\blpage{647}
(\byear{2022}).
\doiurl{10.1007/s10957-021-01981-1}
\end{barticle}
\endbibitem

\bibitem{Pap3}
\begin{barticle}
\bauthor{\bsnm{Som}, \binits{K.}},
\bauthor{\bsnm{Vetrivel}, \binits{V.}}:
\batitle{Global well-posedness of set-valued optimization with application to uncertain problems}.
\bjtitle{J. Global Optim.}
\bvolume{85}(\bissue{2}),
\bfpage{511}--\blpage{539}
(\byear{2023}).
\doiurl{10.1007/s10898-022-01208-1}
\end{barticle}
\endbibitem

\bibitem{HamHey10}
\begin{barticle}
\bauthor{\bsnm{Hamel}, \binits{A.H.}},
\bauthor{\bsnm{Heyde}, \binits{F.}}:
\batitle{Duality for set-valued measures of risk}.
\bjtitle{SIAM J. Financial Math.}
\bvolume{1}(\bissue{1}),
\bfpage{66}--\blpage{95}
(\byear{2010}).
\doiurl{10.1137/080743494}
\end{barticle}
\endbibitem

\bibitem{HM1}
\begin{barticle}
\bauthor{\bsnm{Hamel}, \binits{A.H.}},
\bauthor{\bsnm{Heyde}, \binits{F.}},
\bauthor{\bsnm{Rudloff}, \binits{B.}}:
\batitle{Set-valued risk measures for conical market models}.
\bjtitle{Mathematics and Financial Economics}
\bvolume{5}(\bissue{1}),
\bfpage{1}--\blpage{28}
(\byear{2011}).
\doiurl{10.1007/s11579-011-0047-0}
\end{barticle}
\endbibitem

\bibitem{Boris2010}
\begin{bchapter}
\bauthor{\bsnm{Bao}, \binits{T.Q.}},
\bauthor{\bsnm{Mordukhovich}, \binits{B.S.}}:
\bctitle{Set-valued optimization in welfare economics}.
In: \bbtitle{Advances in Mathematical Economics. {V}olume 13}.
\bsertitle{Adv. Math. Econ.},
vol. \bseriesno{13},
pp. \bfpage{113}--\blpage{153}.
\bpublisher{Springer},
\blocation{Tokyo}
(\byear{2010}).
\doiurl{10.1007/978-4-431-99490-9\_5}
\end{bchapter}
\endbibitem

\bibitem{H}
\begin{barticle}
\bauthor{\bsnm{Hamel}, \binits{A.H.}},
\bauthor{\bsnm{L\"{o}hne}, \binits{A.}}:
\batitle{A set optimization approach to zero-sum matrix games with multi-dimensional payoffs}.
\bjtitle{Math. Methods Oper. Res.}
\bvolume{88}(\bissue{3}),
\bfpage{369}--\blpage{397}
(\byear{2018}).
\doiurl{10.1007/s00186-018-0639-z}
\end{barticle}
\endbibitem

\bibitem{Zemkohoillposed}
\begin{barticle}
\bauthor{\bsnm{Zemkoho}, \binits{A.B.}}:
\batitle{Solving ill-posed bilevel programs}.
\bjtitle{Set-Valued Var. Anal.}
\bvolume{24}(\bissue{3}),
\bfpage{423}--\blpage{448}
(\byear{2016}).
\doiurl{10.1007/s11228-016-0371-x}
\end{barticle}
\endbibitem

\bibitem{pileckathesis}
\begin{botherref}
\oauthor{\bsnm{Pilecka}, \binits{M.}}:
Set-valued optimization and its application to bilevel optimization.
PhD thesis,
TU Bergakademie Freiberg,
Freiberg, Germany
(2015)
\end{botherref}
\endbibitem

\bibitem{Pap6}
\begin{botherref}
\oauthor{\bsnm{Som}, \binits{K.}},
\oauthor{\bsnm{Dutta}, \binits{J.}},
\oauthor{\bsnm{Thirumulanathan}, \binits{D.}}:
Bilevel programming problems: A view through set-valued optimization.
(accepted for publication)
\end{botherref}
\endbibitem

\bibitem{IDE}
\begin{botherref}
\oauthor{\bsnm{Ide}, \binits{J.}},
\oauthor{\bsnm{K\"{o}bis}, \binits{E.}},
\oauthor{\bsnm{Kuroiwa}, \binits{D.}},
\oauthor{\bsnm{Sch\"{o}bel}, \binits{A.}},
\oauthor{\bsnm{Tammer}, \binits{C.}}:
The relationship between multi-objective robustness concepts and set-valued optimization.
Fixed Point Theory Appl.,
2014--8320
(2014).
\doiurl{10.1186/1687-1812-2014-83}
\end{botherref}
\endbibitem

\bibitem{CreKurRoc17}
\begin{barticle}
\bauthor{\bsnm{Crespi}, \binits{G.P.}},
\bauthor{\bsnm{Kuroiwa}, \binits{D.}},
\bauthor{\bsnm{Rocca}, \binits{M.}}:
\batitle{Quasiconvexity of set-valued maps assures well-posedness of robust vector optimization}.
\bjtitle{Ann. Oper. Res.}
\bvolume{251}(\bissue{1-2}),
\bfpage{89}--\blpage{104}
(\byear{2017}).
\doiurl{10.1007/s10479-015-1813-9}
\end{barticle}
\endbibitem

\bibitem{Pap1}
\begin{barticle}
\bauthor{\bsnm{Som}, \binits{K.}},
\bauthor{\bsnm{Vetrivel}, \binits{V.}}:
\batitle{On robustness for set-valued optimization problems}.
\bjtitle{J. Global Optim.}
\bvolume{79}(\bissue{4}),
\bfpage{905}--\blpage{925}
(\byear{2021}).
\doiurl{10.1007/s10898-020-00959-z}
\end{barticle}
\endbibitem

\bibitem{postolm1986vectorial}
\begin{barticle}
\bauthor{\bsnm{Postolic\u{a}}, \binits{V.}}:
\batitle{Vectorial optimization programs with multifunctions and duality}.
\bjtitle{Ann. Sci. Math. Qu\'{e}bec}
\bvolume{10}(\bissue{1}),
\bfpage{85}--\blpage{102}
(\byear{1986})
\end{barticle}
\endbibitem

\end{thebibliography}


\end{document}